\documentclass[12pt,reqno]{amsart}

\usepackage{xcolor}
\usepackage{times}
\usepackage{graphicx}
\usepackage{amsmath,amsthm,amsfonts,amscd,amssymb, MnSymbol}
\usepackage{pst-node}
\usepackage{pst-plot}

\usepackage[all]{xy}

\usepackage{tikz, braids}
\usetikzlibrary{decorations.pathreplacing, arrows}

\newtheorem{thm}{Theorem}[section]
\newtheorem{cor}[thm]{Corollary}
\newtheorem{defn}[thm]{Definition}
\newtheorem{lem}[thm]{Lemma}
\newtheorem{rem}[thm]{Remark}
\newtheorem{prop}[thm]{Proposition}

{ \theoremstyle{remark} }

\numberwithin{thm}{section}
\numberwithin{equation}{section}

\newcommand{\myop}[1]{\operatorname{#1}}

\newcommand{\Span}{\myop{span}}

\newcommand{\Real}{\mathbb R}

\newcommand{\eps}{\varepsilon}

\newcommand{\la}{\langle}
\newcommand{\ra}{\rangle}

\newcommand{\A}{\mathcal{A}}

\newcommand{\B}{\mathcal{B}}

\newcommand{\Comp}{\mathbb{C}}

\newcommand{\D}{\mathcal{D}}

\newcommand{\Fc}{\mathcal{F}}
\newcommand{\g}{\mathfrak{g}}
\newcommand{\h}{\mathfrak{h}}
\newcommand{\Hi}{\mathcal{H}}

\newcommand{\n}{\mathbb{N}}

\newcommand{\tor}{\mathbb{T}}

\newcommand{\z}{\mathbb{Z}}

\newcommand{\Om}{\Omega}
\newcommand{\om}{\omega}

\begin{document}

\title[Analytic subalgebras of Beurling-Fourier algebras]{Analytic subalgebras of Beurling-Fourier algebras and complexification of Lie groups}

\author{Heon Lee}
\address{Heon Lee, Department of Mathematical Sciences and the Research Institute of Mathematics, Seoul National University, Gwanak-ro 1, Gwanak-gu, Seoul 08826, Republic of Korea}
\email{heoney93@gmail.com}

\author {Hun Hee Lee}
\address{Hun Hee Lee, Department of Mathematical Sciences and the Research Institute of Mathematics, Seoul National University, Gwanak-ro 1, Gwanak-gu, Seoul 08826, Republic of Korea}
\email{hunheelee@snu.ac.kr}

\maketitle

\begin{abstract}
In this paper, we focus on how we can interpret the actions of the elements in the Gelfand spectrum of a weighted Fourier algebra on connected Lie groups. They can be viewed as evaluations on specific points of the complexification of the underlying Lie group by restricting to a particular dense subalgebra, which we call an analytic subalgebra. We first introduce an analytic subalgebra allowing a ``local" solution for general connected Lie groups. We will demonstrate that a ``global" solution is also possible for connected, simply connected and nilpotent Lie groups through a different choice of an analytic subalgebra. Finally, we examine the case of the $ax+b$-group as an example of a non-nilpotent, non-unimodular Lie group with a ``global" solution.
\end{abstract}

\section{Introduction}

Variations using ``weights" are common techniques in numerous areas of mathematics. The same principle can be applied to the field of abstract harmonic analysis, where we study, for example, operator/function algebras on locally compact groups.
Recall that for a locally compact group $G$ the {\em Fourier algebra} $A(G)$ is defined as the space of all coefficient functions of the {\em left regular representation} $\lambda: G\to B(L^2(G))$:
    $$A(G) := \{u\in C_0(G): u(s) = \la \lambda(s) f, g\ra,\; s\in G,\; f,g\in L^2(G)\}.$$
The space $A(G)$ is equipped with the norm
    $$\|u\|_{A(G)} = \inf\{\|f\|_2\|g\|_2: u(s) = \la \lambda(s) f, g\ra \}$$
and it is a commutative Banach algebra with respect to point-wise multiplication.
The commutativity of $A(G)$ naturally leads us to look at ${\rm Spec}A(G)$, its {\em Gelfand spectrum}, which is known to be homeomorphic to $G$. In other words, we have ${\rm Spec}A(G) \cong G$ as topological spaces. See \cite[Section 2.3]{Kaniuth} for the details.
A typical proof of this fact is accompanied by the following interpretation, namely non-zero multiplicative linear functionals on $A(G)$ are nothing but evaluations on the points of $G$.

On the other hand, the Fourier algebra $A(G)$ is the dual object of the {\em convolution algebra} $L^1(G)$. This duality can be explicitly seen when the group $G$ is abelian;
more precisely, we have the Banach algebra isomorphism $A(G) \cong L^1(\widehat{G})$, where $\widehat{G}$ refers to the {\em dual group} of $G$. The dual object $L^1(G)$ has the ``weighted" version $L^1(G,w)$ for some {\em weight function}
(i.e. {\em submultiplicative function}) $w: G\to (0,\infty)$, which we call the {\em Beurling algebra}. There have been several attempts (\cite{LST12, LS12}) to develop a corresponding dual theory, namely ``weighted" versions of Fourier algebras $A(G,W)$ under the name of Beurling-Fourier algebras. 
The precise definition of a ``weight" $W$ is more involved than the concept of a weight function, so we refer to Section \ref{subsec-weights} for the details.

Since a Beurling-Fourier algebra $A(G,W)$ is still commutative, it is natural to be interested in its Gelfand spectrum ${\rm Spec}A(G,W)$. 
The results from \cite{LST12} and \cite{GLLST} showed that the structure of ${\rm Spec}A(G,W)$ for a Lie group $G$ is closely related to the {\em complexification} of $G$ and were able to determine ${\rm Spec}A(G,W)$ focusing on concrete examples. 
Finally, \cite{GT} succeeded in determining ${\rm Spec}A(G,W)$ in the case of a general connected Lie group $G$ and a weight $W$ coming from an abelian Lie subgroup using an abstract version of complexification (see Section \ref{subsec-abs-complexification} for details).

However, we still need more information on how to interpret the actions of the elements of ${\rm Spec}A(G,W)$ as certain evaluations.
Let us re-examine the case $G = \Real$ for a better illustration. Since $G = \Real$ is abelian, our model $A(G,W)$ becomes the Beurling algebra $L^1(\widehat{G},w) = L^1(\Real,w)$ for some weight function $w: \widehat{G} \cong \Real \to (0,\infty)$. One possible approach is to find a subalgebra $\A$ in $L^1(\widehat{G},w)$ satisfying the following:
\begin{enumerate}
\item the topological algebra $\A$ is continuously embedded in $L^1(\widehat{G},w)$ with a dense image,
\item its Gelfand spectrum (i.e. the space of all continuous characters) ${\rm Spec}\A \cong \Comp$, where the latter can be understood as a complexification of the real Lie group $\Real$ and
\item all elements in ${\rm Spec}L^1(\widehat{G},w)$ act on $\A$ as evaluations on the points of $\Comp$.
\end{enumerate}
In fact, the choice $\A = \Fc^\Real(C^\infty_c(\Real))$ satisfy all the above conditions.
Moreover, the algebra $\A = \Fc^\Real(C^\infty_c(\Real))$ is independent of the choice of the weight function $w$. Here, $\Fc^\Real$ refers to the Fourier transform in $\Real$. Condition (1) allows us to observe that
$${\rm Spec}L^1(\widehat{G},w) \subseteq {\rm Spec}\A$$
and Condition (2) and (3) lead us to specific forms of the elements of ${\rm Spec}L^1(\widehat{G},w)$ (in this case, exponential functions) so that we only need to check its boundedness with respect to the $L^1(\widehat{G},w)$-norm to determine ${\rm Spec}L^1(\widehat{G},w)$. See \cite[Section 1.1]{GLLST} for more details.

In \cite{GLLST} the authors, in principle, followed the same approach for the case $G=\mathbb{H}$, the Heisenberg group, with a certain choice of weight $W$, namely finding a subalgebra $\A$ in $A(G,W)$ satisfying the following:
\begin{enumerate}
\item the topological algebra $\A$ is continuously embedded in $A(G,W)$ with a dense image,
\item its Gelfand spectrum ${\rm Spec}\A \cong \mathbb{H}_\Comp$, where the latter is the {\em universal complexification} of $\mathbb{H}$ and
\item all elements in ${\rm Spec}A(G,W)$ act on $\A$ as evaluations on the points of $\mathbb{H}_\Comp$.
\end{enumerate}
The specific choice made in \cite{GLLST} was $\A = \Fc^{\Real^3}(C^\infty_c(\Real^3))$ based on the identification $\mathbb{H}\cong \Real^3$ as topological spaces that share the same translation-invariant measure, the Lebesgue measure.
The above choice of $\A$ looks quite similar to the case of $\Real$. However, detailed proofs for the required properties can only be completed with the help of additional subalgebras and subspaces (see \cite[Definition 6.4]{GLLST}) whose choice was quite technical. 
Furthermore, other groups, including the Euclidean motion group $E(2)$, have been examined in \cite{GLLST}, 
where the choice of $\A$ was fundamentally different from the case of $\mathbb{H}$. The technicality of the choice of the subalgebra $\A$ and its companions, as well as the dependence on the group, was the main obstacle in \cite{GLLST} in developing a general theory.

The main goal of this paper is to continue the approach in \cite{GLLST} to deal with the general cases covered in \cite{GT}. More precisely, for any connected Lie group and a weight $W$ coming from a connected abelian Lie subgroup of $G$, we hope to construct a subalgebra $\A$ of $A(G,W)$ that satisfies the following:
\begin{enumerate}
\item[(A1)] the topological algebra $\A$ is continuously embedded in $A(G,W)$ with a dense image,
\item[(A2)] its Gelfand spectrum ${\rm Spec}\A \cong G_\Comp$, where the latter is a complexification of $G$ and
\item[(A3)] all elements in ${\rm Spec}A(G,W)$ act on $\A$ as evaluations on the points of $G_\Comp$.
\end{enumerate}
Although completion of the above program in full generality is not available to us, we will find a subalgebra $\A$ that satisfies (A1) and a ``local" version (A3') of (A3) for certain choice of weight $W$ as follows in Section \ref{sec-local-sol}.

\begin{enumerate}
\item[(A3')] all elements in ${\rm Spec}A(G,W)$ act on $\A$ as evaluations on the points of an open neighborhood of $G$ within $G_\Comp$.
\end{enumerate}
Note that the main conclusion in Section \ref{sec-local-sol}, Theorem \ref{thm-eval-action}, depends on the results of \cite{GT} through Theorem \ref{thm-spectrum-extended-weight-general} below.

For connected, simply connected and nilpotent Lie groups $G$ we will find a subalgebra $\A$ that satisfies (A1), a ``partial" version (A2') and (A3") of (A2) and (A3), respectively, as follows in Section \ref{sec-global-sol} for any weight $W$ extended from an abelian subgroup.

\begin{enumerate}
\item[(A2')] its Gelfand spectrum ${\rm Spec}\A \cong G^\h$, where $G^\h$ is a properly embedded submanifold of $G_\Comp$ to be introduced in Section \ref{sec-global-sol} and
\item[(A3")] all elements in ${\rm Spec}A(G,W)$ act on $\A$ as evaluations on the points of $G^\h$.
\end{enumerate}
Note that the results in Section \ref{sec-global-sol} do not depend on the results of \cite{GT}.

This paper is organized as follows.
In Section \ref{sec-prelim} we collect basic materials on locally convex spaces, unbounded operators, and Lie theory. In Section \ref{subsec-weights} we review Fourier algebras and the concept of weight, which is followed by the description of Beurling-Fourier algebras in Section \ref{subsec-BF-alg}. The concept of abstract complexification of Lie groups will be presented in Section \ref{subsec-abs-complexification}.

Section \ref{sec-anal-subalg-density} focuses on the technical definition of our first model of analytic subalgebras using analytic vectors for the left regular representation. A ``local" solution for general connected Lie groups will be presented in Section \ref{sec-local-sol} after we clarify the analytic extensions that we focus on.

In Section \ref{sec-global-sol} we move our attention to connected, simply connected, and nilpotent Lie groups and we establish a ``global" solution via a different choice of analytic subalgebra. In the last section, we will examine the case of $ax+b$-group as an example of non-nilpotent, non-unimodular Lie group with a ``global" solution.

\section{Preliminaries}\label{sec-prelim}

\subsection{Locally convex spaces: tensor products, quotients and vector-valued integrals}

The analytic subalgebras we will construct in this paper lie beyond the scope of Banach spaces. We collect essential features (mainly from \cite[Section 18]{Kothe83}, \cite[Chapter III]{SchWol99} and \cite[Chapter 2]{Hor66} if not specified) of {\em topological vector spaces}, which will be frequently used later on. 

Recall that a Hausdorff topological vector space $X$ over $\Comp$ is called {\em locally convex} if it has a base of neighborhoods of $0$ consisting of convex sets, or equivalently, there is a family of semi-norms $\{p_\alpha\}_{\alpha\in A}$ on $X$ that determines its topology in a canonical way.
We say that a locally convex (topological vector) space $X$ is a {\em Fr\'echet space} if it is complete and metrizable. In this case we may take a countable family of semi-norms $\{p_n\}_{n\ge 0}$.

Let $X$ be a locally convex space whose topology is given by a family of semi-norms $\{p_\alpha\}_{\alpha\in A}$. A subspace $H$ of $X$ inherits a locally convex space structure via the subspace topology or,
equivalently, the topology given by the restrictions of the seminorms $\{(p_\alpha)|_H\}_{\alpha\in A}$.
When $H$ is closed, the associated quotient space $X/H$ is also locally convex and its topology is given by a family of induced semi-norms $\{p'_\alpha\}_{\alpha\in A}$, where
$$p'_\alpha(x+H) := \inf_{h\in H}p_\alpha(x+h),\;\;x\in X.$$
When $X$ is a Fr\'echet space, then $X/H$ is a Fr\'echet space as well.

For a family $\{X_\alpha\}_{\alpha \in A}$ of locally convex spaces, the topological product $X= \prod_{\alpha \in A}X_\alpha$ is locally convex again. More precisely, if $\{p^\alpha_\beta\}_\beta$ is a family of semi-norms that give the topology of $X_\alpha$, $\alpha \in A$, then the product topology of $X$ is given by the semi-norms
    $$\tilde{p}^\alpha_\beta(x) := p^\alpha_\beta(x_\alpha),\;\;x = (x_\alpha) \in X.$$

Note that the topological products of Banach spaces are always complete, which allows us to define the completion of a locally convex space. Recall that a locally convex space $X$ can be mapped onto a subspace $\tilde{X}$ of a topological product of Banach spaces through a homeomorphic linear isomorphism.
Then, the closure of $\tilde{X}$ in the product is the {\em completion} of $X$, which we denote by $\overline{X}$. When the topology of $X$ is given by a family of semi-norms $\{p_\alpha\}_{\alpha\in A}$, then the topology of $\overline{X}$ is given by a corresponding family of semi-norms $\{\bar{p}_\alpha\}_{\alpha\in A}$, where $\bar{p}_\alpha$ is the unique continuous extension of $p_\alpha$ to $\overline{X}$, making $\overline{X}$ a locally convex space.

Suppose now that for a family $\{X_\alpha\}_{\alpha \in A}$ of locally convex spaces, there are a locally convex space $E$ and injective continuous maps $J_\alpha :X_\alpha \to E$, $\alpha \in A$. This allows us to consider the intersection $\bigcap_{\alpha \in A} X_\alpha$ given by
    $$\bigcap_{\alpha \in A} X_\alpha := \{x_\alpha \in \prod_{\alpha \in A}X_\alpha : J_\alpha(x_\alpha) = J_\beta(x_\beta),\;\; \forall \alpha,\beta \in A \}.$$
Note that the choice of $E$ and $J_\alpha$ is typically apparent in the context, and the map $J_\alpha$ is usually omitted. The intersection $\bigcap_{\alpha \in A} X_\alpha$ carries the locally convex space structure as a subspace of $\prod_{\alpha \in A}X_\alpha$.


Let $p$ and $q$ be semi-norms of $X$ and $Y$, respectively. Then, the {\em tensor product semi-norm} $p\otimes q$ on the {\em algebraic tensor product} $X\odot Y$ is given by
    $$p\otimes q(u) := \inf\left\{\sum^n_{i=1} p(x_i)q(y_i): u = \sum^n_{i=1} x_i \otimes y_i \right\},\;\; u\in X\odot Y.$$
For two locally convex spaces $X$ and $Y$ whose topologies are generated by families of semi-norms $\{p_\alpha\}_{\alpha\in A}$ and $\{q_\beta\}_{\beta\in B}$, respectively, the family of semi-norms $\{p_\alpha \otimes q_\beta\}_{\alpha\in A, \beta \in B}$ give rise to a locally convex topology on $X\odot Y$.
We call its completion the {\em complete projective tensor product} of $X$ and $Y$, which will be denoted by $X\otimes_\pi Y$. See \cite[Section III.6]{SchWol99} for the details.
When $X$ and $Y$ are Frech\'{e}t spaces, then $X\otimes_\pi Y$ is also a Fr\'{e}chet space. 

Let $T: X\to Y$ be a linear map between two locally convex spaces $X$ and $Y$ whose topologies are generated by families of semi-norms $\{p_\alpha\}_{\alpha\in A}$ and $\{q_\beta\}_{\beta\in B}$, respectively. It is well-known that $T$ is continuous if and only if for every $\beta$ there are $\alpha_1,\cdots, \alpha_n \in A$ and $C>0$ such that
    $$q_\beta(T(x))\le C(p_{\alpha_1}(x) + \cdots + p_{\alpha_n}(x)),\;\;x\in X.$$
We say that two locally convex spaces are {\em isomorphic} if there is a bi-continuous linear bijection between them.

A Hausdorff topological vector space $\A$ is called a {\em topological algebra} if there is an associative algebra multiplication $m: \A \times \A \to \A$, which is bilinear and separately continuous. We say that $\A$ is a {\em Fr\'{e}chet algebra}
if $\A$ is a Fr\'{e}chet space.
Note that (\cite[Chapter VII, Proposition 1]{Wa71}) the multiplication map of a Fr\'{e}chet algebra is automatically jointly continuous.

\subsubsection{Pettis integral}

Let $X$ be a fixed Fr\'echet space and $(\Om, \Sigma, \mu)$ be a measure space. A function $f:\Om \to X$ is called {\em weakly measurable} if $(f(\cdot), \varphi)$ is $\Sigma$-measurable for any $\varphi\in X^*$, or equivalently (see \cite[Theorem 1]{Th75}) $(f(\cdot), \varphi)$ is $\Sigma$-measurable for any $\varphi\in S$ for some total (i.e. separates points of $X$) subset $S$ of $X^*$. A weakly measurable function $f:\Om \to X$ is called {\em Pettis integrable} if (1) $(f(\cdot), \varphi)$ is $\mu$-integrable for any $\varphi\in X^*$ and (2) for any $A\in \Sigma$ there is $x_A\in X$ such that
    $$(x_A, \varphi) = \int_A (f, \varphi)\,d\mu,\;\; \varphi\in X^*.$$
We call $x_A$ the {\em Pettis integral} of $f$ on $A$ and denote it by $\int_A f\,d\mu$. For a lower semi-continuous semi-norm $p$ on $X$ and $A\in \Sigma$ the function $p(f(\cdot))$ is $\Sigma$-measurable and we have
    $$p(\int_A f\,d\mu)\le \int_A p(f)\,d\mu.$$

\subsection{Unbounded operators}
We collect some of the basic materials on unbounded operators excerpted from  \cite[Chapter~1--7]{Schm12}.

Recall that a linear map $T: \D(T) \subset \Hi \to \Hi$ acting on a Hilbert space $\Hi$ with domain $\D(T)$ is said to be {\em closed}
if its graph ${\rm gra}(T) = \{(x,T(x)): x\in \D(T)\}$ is closed in the direct sum $\Hi\oplus \Hi$ of Hilbert spaces.
We say that a linear map $S:\D(S) \subseteq \Hi \to \Hi$ is an {\em extension} of $T$ (we write $T\subseteq S$) if $\D(T) \subseteq \D(S)$ and $S|_{\D(T)} = T$.
We say that $T$ is {\em closable} if it has a closed extension. In this case, we denote the smallest closed extension (with respect to the inclusion of the corresponding domain) by $\overline{T}$, the {\em closure} of $T$.

A densely defined operator $T: \D(T)\subseteq \Hi \to \Hi$ (that is, $\D(T)$ is dense in $\Hi$) has a well-defined adjoint (linear and closed) operator $T^*: \D(T^*)\subseteq \Hi \to \Hi$. We say that such a map $T$ is {\em self-adjoint} if $T = T^*$, in which case $T$ is also closed. We say that such a map $T$ is {\em essentially self-adjoint} if the closure $\overline{T}$ is self-adjoint.
A densely defined operator $T: \D(T) \subseteq \Hi \rightarrow \Hi$ is called {\em normal} if $T$ is closed and $TT^* = T^* T$, in which case we also have $\D(T) = \D(T^*)$.

A normal operator $T: \D(T)\subseteq \Hi \to \Hi$ can be written as a spectral integral. More precisely, there is a projection-valued spectral measure
$$E = E_T: \B(\Comp) \to B(\Hi)$$ such that $\displaystyle T = \int_\Comp \lambda \,dE(\lambda)$, where $\B(\Comp)$ refers to the Borel $\sigma$-algebra on $\Comp$. For $v,w \in \Hi$ we get a complex measure $E_{v,w}$ given by
$$E_{v,w}(A) = \la E(A)v, w \ra,\;\; A\in \B(\Comp).$$
The operator $T$ is then self-adjoint if and only if the spectral measure $E$ is supported on the real line $\Real$. This integral representation leads to a single-variable Borel functional calculus. In fact, for any Borel measurable function $f:\Comp \to \Comp$ we define
$$f(T) := \int_\Comp f(\lambda) \,dE(\lambda)$$
with domain $\displaystyle \D(f(T)) =\{v\in \Hi: \int_\Comp |f(\lambda)|^2 \,dE_v(\lambda)<\infty\}$, where we simply write $E_v = E_{v,v}$. In particular, the ``absolute value" $|T|$ is defined by $\displaystyle |T| := \int_\Comp |\lambda| \,dE(\lambda)$. We can easily see that $|T|^n = |T^n|$ and $\|T^n v\| = \|\big| T \big|^n v\|$ for all $n\ge 1$ and $v\in \D(T^n) = \D(|T|^n)$. Let us record an elementary lemma regarding series expressions for operators defined by functional calculus.

\begin{lem}\label{lem-abs-conv}
Let $T :\D(T)\subseteq \Hi \to \Hi$ be a normal operator. For a vector $v\in \bigcap_{n\ge 1}\D(|T|^n)$ with $\displaystyle \sum_{n\ge 0}\frac{\|\big|T\big|^nv\|}{n!}<\infty$, we have $v\in \D(e^T)$ and 
    $$e^T v = \sum_{n\ge 0}\frac{T^nv}{n!},$$
where the series converges absolutely in $\Hi$.
\end{lem}
\begin{proof}
For such a vector $v\in \Hi$ we first show that $v\in \D(e^{|T|})$.
Let $E$ be the spectral measure on $\Comp$ for the operator $T$. We apply the monotone convergence theorem to get
\begin{align*}
    \int_\Comp e^{2|\lambda|}\, dE_v(\lambda)
    & = \lim_{n\to \infty} \int_\Comp (\sum^n_{k=0}\frac{|\lambda|^k}{k!})^2\, dE_v(\lambda)\\
    & = \lim_{n\to \infty} \|\sum^n_{k=0}\frac{|T|^kv}{k!}\|^2 \le (\sum_{n\ge 0}\frac{\|\big|T\big |^nv\|}{n!})^2 <\infty,
\end{align*}
which means that $v\in \D(e^{|T|})$. Then, for any $w\in \Hi$ we have $e^{|\lambda|} \in L^1(E_{v,w})$, which allows us to use the Lebesgue dominated convergence theorem for
    $$\la e^T v, w \ra = \int_\Comp e^\lambda\, dE_{v,w}(\lambda) = \lim_{n\to \infty} \la \sum^n_{k=0}\frac{T^kv}{k!}, w \ra$$
so that we have
    $$e^T v = \sum_{n\ge 0}\frac{T^nv}{n!},$$
where the series converges absolutely since $\|\big| T \big|^nv\| = \|T^nv\|, n\ge 0$.
\end{proof}

A family of densely defined normal operators $T_1,\dots,T_n$ is said to be {\it strongly commuting} if each pair of
elements $E_{T_i}(A)$ and $E_{T_j}(B)$ commutes, where $A,B\in \B(\Comp)$. In this case, there is a spectral measure $E: \B(\Comp^n) \to B(\Hi)$ such that $\displaystyle T_i = \int_{\Comp^n} \lambda_i \,dE(\lambda_1, \cdots, \lambda_n)$, $1\le i\le n$. Similarly, we have a multi-variable Borel functional calculus, namely, for any Borel measurable function $g:\Comp^n \to \Comp$ we define
    $$g(T_1,\cdots, T_n) := \int_{\Comp^n} g(\lambda_1, \cdots, \lambda_n) \,dE(\lambda_1, \cdots, \lambda_n).$$

Let $T: \D(T) \subseteq \Hi \rightarrow \Hi$ and $S: \D(S) \subseteq \Hi \rightarrow \Hi$ be two densely defined closable operators. Then, the densely defined operator
    $$T \odot S : \D(T) \odot \D(S) \subseteq \Hi \otimes \Hi \rightarrow \Hi \otimes \Hi,\;\; \xi \otimes \eta\mapsto T\xi \otimes S\eta$$
is known to be closable, whose closure will be denoted as $T \otimes S$.
It is known that $(T \otimes S)^* = T^* \otimes S^*$, which implies that if $T$ and $S$ are self-adjoint, then $T \otimes S$ is also self-adjoint.

Let $T$ be a self-adjoint operator on $\Hi$ with the associated spectral measure $E: \B (\Real) \rightarrow B(\Hi)$. Then, the map $E \otimes I : \B (\Real) \rightarrow B(\Hi),\;\; A \mapsto E_T (A) \otimes I$ is the spectral measure of the self-adjoint operator $T \otimes I$. Note that if $S$ is another self-adjoint operator on $\Hi$, then $T \otimes I$ and $I \otimes S$ strongly commute and we have $T\otimes S = (T\otimes I)\cdot (I\otimes S)$.

\subsection{Lie groups, algebras and representations}\label{subsec-Lie}

In this paper, we focus on a connected real Lie group $G$ equipped with the associated real Lie algebra $\g \cong \Real^d$ and the {\it exponential map}
    $$\exp : \g \to G.$$
We will choose an abelian Lie subgroup $H$ of G with the associated Lie algebra $\mathfrak{h}$. We fix a basis $\{X_1, \cdots, X_d\}$ for $\g$ such that $\mathfrak{h} = {\rm span}\{X_1, \cdots, X_k\}$ for some $1\le k\le d$. We assume that there is a fixed {\it complexification} $G_\Comp$, which is a complex Lie group that contains $G$ as a real Lie subgroup, and the associated Lie algebra is the complexified Lie algebra $\g_\Comp = \g + i\g$.
For $X = x_1X_1 + \cdots + x_dX_d \in \g_\Comp = \g + i\g$, $(x_j)^d_{j=1}\subseteq \Comp$, we consider its $L^2$-norm
    $$|X| := (|x_1|^2 + \cdots + |x_d|^2)^{\frac{1}{2}}.$$

Let $\pi: G \to B(\Hi_\pi)$ be a strongly continuous unitary representation of $G$, which induces a representation $d \pi$ of the Lie algebra $\g$ as follows: for each $X\in \g$ we define
    $$d \pi(X)v := \frac{d}{dt}\pi(\exp tX)v|_{t=0}$$
for $v\in \Hi^\infty_\pi$, the Fr\'{e}chet space of {\em $C^\infty$-vectors} (i.e. $v\in \Hi^\infty_\pi$ means that $g\in G \mapsto \pi(g)v$ is $C^\infty$) equipped with the semi-norms
     $$\rho^\pi_n(v) = \rho_n(v) := \left(\sum_{1\le j_1,\cdots,j_n\le d}\|d \pi(X_{j_1}\cdots X_{j_n})v\|^2 \right)^{\frac{1}{2}},\;\; v\in \Hi^\infty_\pi,$$
where $d \pi(X_{j_1}\cdots X_{j_n}) = d \pi(X_{j_1})\,\cdots \,d(X_{j_n})$.
Note that we have chosen $L^2$-type of semi-norms $\rho_n$, whilst $L^\infty$-type semi-norms
    $$\tilde{\rho}_n(v) = \sup_{1\le j_1,\cdots,j_n\le d}\|d \pi(X_{j_1}\cdots X_{j_n})v\|$$
were used in the literature (cf. \cite{Goodman_entire}), which are equivalent as follows:
    $$\tilde{\rho}_n(v) \le \rho_n(v) \le d^{\frac{n}{2}}\tilde{\rho}_n(v),\;\; v\in \Hi^\infty_\pi.$$
It is well-known (cf. \cite[Chapter~0]{Taylor_harmonic}) that $i\cdot d \pi (X)$ is an essentially self-adjoint operator acting on $\Hi_\pi$. We write $i\cdot \partial \pi(X)$, the self-adjoint closure of $i\cdot d \pi(X)$ so that $\partial \pi(X)$ is a normal operator acting on $\Hi_\pi$. Stone's theorem says that we have
$$\pi(\exp(X)) = e^{\partial\pi(X)},$$
where the latter is given by the functional calculus.

\subsubsection{Analytic and entire vectors}\label{sec-analytic-entire-vectors}

For $s>0$ and $v\in \Hi^\infty_\pi$ we recall the norm
\begin{equation}\label{eq-seminorm-E}
E^\pi_s(v) = E_s(v) := \sum_{n\ge 0}\frac{s^n}{n!}\rho_n(v).
\end{equation}
Then, the space
    $$\Hi^a_{\pi,t} := \{v\in \Hi^\infty_\pi: E_s(v)<\infty, 0<s<t\}.$$
becomes a Fr\'{e}chet space for each $t>0$.
Indeed, a Cauchy sequence $(v_j)_j \subseteq \mathcal{H}_{\pi, t} ^a $, with respect to the norms $(E_s)_{0 < s < t}$, is Cauchy in the Fr\'{e}chet space $\mathcal{H}_\pi ^\infty$ and hence has a limit $v$ in this space. Fix $0<s<t$. For $^\forall \varepsilon >0$, we can choose a positive integer $N$ such that $i,j \ge N$ implies
    $$\sum_{n=0} ^\infty \frac{s^n}{n!} \rho_n (v_i - v_j) = E_s(v_i - v_j) < \varepsilon.$$
Applying Fatou's lemma with respect to the index $i$, we obtain
    $$\sum_{n=0} ^\infty \frac{s^n}{n!} \rho_n (v - v_j) \leq \varepsilon,$$
proving $E_s(v) < \infty$ and $\displaystyle \lim_{j \rightarrow \infty}E_s(v- v_j) = 0$. Since $0<s<t$ was arbitrary, we see that $v\in \mathcal{H}_{\pi, t}^a$ and $v_j \rightarrow v$ in $\mathcal{H}_{\pi, t}^a$.
Note that the norms
    $$\displaystyle \tilde{E}_s(v) := \sum_{n\ge 0}\frac{s^n}{n!}\tilde{\rho}_n(v)$$
were used more frequently in the literature (cf. \cite{Goodman_entire}) and we have
    $$\tilde{E}_s(v)\le E_s(v)\le \tilde{E}_{s\sqrt{d}}(v),\;\; v\in \Hi^a_{\pi,s\sqrt{d}}.$$

We say that $v\in \Hi$ is an {\em analytic vector} for $\pi$ if $v\in \Hi^a_{\pi,t}$ for some $t>0$.
The space $\Hi^a_{\pi, \infty}$ of all {\em entire vectors} for $\pi$ is defined by
    $$\Hi^a_{\pi, \infty} := \bigcap_{t>0}\Hi^a_{\pi,t},$$
which is a Frech\'{e}t space with respect to the topology given by the family of norms $(E_s)_{s>0}$.

When $\pi=\lambda$, the left regular representation on $G$, we will 
simply write $\Hi^\infty$, $\Hi^a_t$ and $\Hi^a_\infty$ instead of $\Hi^\infty_\lambda$, $\Hi^a_{\lambda,t}$ and $\Hi^a_{\lambda,\infty}$, respectively.

\begin{rem}
\begin{enumerate}
    \item Our choices of using the $L^2$-type of (semi-)norms $\rho_n$ and $|\cdot|$ on $\g$ are suitable for the definition and algebraic properties of the analytic subalgebras to be defined in Section \ref{sec-anal-subalg-density}.
    \item For any unitary representation $\pi: G \to B(\Hi_\pi)$ of $G$ the space $\Hi^a_{\pi,t}$ is known to be dense in $\Hi_\pi$ for small enough $t>0$. See \cite[Theorem~2.1]{Goodman_entire}.
\end{enumerate}

\end{rem}

The representation $d \pi$ of $\g$ on $\Hi^\infty _\pi$ can easily be extended to the complexification $\g_\Comp$ by
    $$d \pi(Z) := d \pi(X) + i \cdot d \pi(Y)$$
for $Z = X + i Y \in \g_\mathbb{C}$ with $X, Y \in \g$.

\begin{prop}\label{prop-analytic evaluation-domain}
\begin{enumerate}
    \item For $Z \in \g_\Comp$ and $v\in \Hi^\infty_\pi$ we have
    \begin{equation}\label{eq-rho-comp}
        \rho_n(d \pi(Z)^m v) \leq |X|^m \rho_{n+m}(v).
    \end{equation}

    \item The space $\Hi_{\pi,t} ^a $ $(t>0)$ is invariant under $d \pi (Z)$ for any $Z \in \g_\Comp$.

    \item For $Z\in \g_\Comp$ with $|Z|<t$ and $v \in \Hi^a_{\pi, t}$ the series $\displaystyle \sum_{n=0} ^\infty \frac{1}{n!} d\pi(Z)^n v$
converges absolutely.
\end{enumerate}
\end{prop}
\begin{proof}
(1) Let $Z = z_1 X_1 + \cdots + z_d X_d$. The first assertion follows from
\begin{align*}
\lefteqn{\rho_n (d \pi(Z) v) }\\
& = \left( \sum_{1 \le j_1 , \cdots, j_n \le d} \| d \pi (X_{j_1} \cdots X_{j_n} Z) v \|^2 \right)^{\frac{1}{2}}\\
& \leq \left( \sum_{1 \le j_1 , \cdots, j_n \le d} \left[ \sum_{i=1} ^d |z_i| \cdot \| d \pi (X_{j_1} \cdots X_{j_n} X_i ) v \| \right]^2 \right)^{\frac{1}{2}}\\
& \leq |Z| \; \rho_{n+1} (v),
\end{align*}
where we have used the Cauchy-Schwartz inequality with respect to the index $i$ at the last inequality.

\vspace{0.5cm}

(2) For $v \in \Hi_{\pi, t} ^a$ we fix $0 < s < t$. We choose $s< r < t$ and $C>0$ such that $(n+1) s^n< C r^{n+1}$ for all $n \ge 1.$
Then we have
\begin{align*}
    \lefteqn{\sum_{n=0} ^\infty \frac{s^n}{n!} \rho_n \big( d \pi (Z) v \big) \leq |Z| \sum_{n=0} ^\infty \frac{s^n}{n!} \rho_{n+1} (v)} \\
    \;\;&= |Z| \sum_{n=0} ^\infty \frac{s^n(n+1)}{(n+1)!} \rho_{n+1} (v) < |Z| \sum_{n=0} ^\infty \frac{(n+1)s^n}{(n+1)!} \rho_{n+1} (v) \\
    \;\;\;\;&= C |Z| \sum_{n=0} ^\infty \frac{r^{n+1}}{(n+1)!} \rho_{n+1} (v) < \infty,
\end{align*}
which shows $d \pi (Z) \left( \Hi_{\pi, t} ^a \right) \subseteq \left( \Hi_{\pi,t} ^a \right)$.

\vspace{0.5cm}

(3)
We only need to note the following.
\begin{align*}
\lefteqn{\sum_{n=0} ^\infty \frac{1}{n!} \| d\pi(Z)^n v \| } \\
& \leq \sum_{n=0} ^\infty \frac{1}{n!} \sum_{1 \leq j_1 , \cdots , j_n \leq d} |z_{j_1}|\, \cdots \,| z_{j_n} | \cdot \| d \pi (X_{j_1} \cdots X_{j_d} )  v \|\\
&  \leq \sum_{n=0} ^\infty \frac{1}{n!} \left(\sum_{1 \leq j_1 , \cdots , j_n \leq d} (|z_{j_1} |\, \cdots \,| z_{j_d} |)^2 \right)^{\frac{1}{2}} \left( \sum_{1 \leq j_1 , \cdots , j_n \leq d} \| d\pi(X_{j_1} \cdots X_{j_n}) v \|^{2} \right)^{\frac{1}{2}} \\
& = \sum_{n=0} ^\infty \frac{1}{n!} |Z| ^n \rho_{n} (v) = E_{|Z|} (v) < \infty.
\end{align*}
\end{proof}

Note that \eqref{eq-rho-comp} is a replacement for the inequality \cite[(1.2)]{Goodman_entire}.
Now we define the operator $e^{d\pi(Z)}$ (as in \cite[Proposition 2.2]{Goodman_entire}) by
\begin{equation}\label{eq-rep-comp}
e^{d\pi(Z)}v := \sum_{n=0} ^\infty \frac{1}{n!} d\pi(Z)^n v
\end{equation}
for $Z\in \g_\Comp$ with $|Z|<t$ and $v \in \Hi^a_{\pi, t}$. When $Z$ is actually an element of $\g$, the above definition $e^{d\pi(Z)}v$ coincides with $\pi(\exp(Z))v = e^{\partial\pi(Z)}v$ by Lemma \ref{lem-abs-conv}. 
The following is a replacement for \cite[Proposition 2.2, Corollary 2.1]{Goodman_entire}.

\begin{prop}\label{prop-holomorphic-ext}
For $Z\in \g_\Comp$ with $|Z|<t$ we have
    $$e^{d\pi(Z)}: \Hi^a_{\pi, t} \to \Hi^a_{\pi, t - |Z|}$$
continuously.
For a vector $v \in \Hi^a_{\pi, t}$
the map
	$$Z \mapsto e^{d \pi(Z)} v \in \Hi_\pi$$
is holomorphic on the region $|Z|<t$ in $\g_\Comp$.
\end{prop}
\begin{proof}
Fix $0 < s < t- |Z|$ and let $v \in \Hi^a_{\pi, t}$. We first check that the series $\displaystyle e^{d \pi(Z)} v=\sum_{n=0} ^\infty \frac{1}{n!} d \pi(Z)^n v$ converges in $\mathcal{H}_\pi ^\infty$. Indeed, for any $m\ge 1$ we know that $d \pi (Z)^m v \in \mathcal{H}^a _{\pi, t}$ by Proposition~\ref{prop-analytic evaluation-domain} (2).
Using \eqref{eq-rho-comp}, we obtain
\begin{align*}
\lefteqn{\sum_{n=0} ^\infty \frac{1}{n!} \rho_m \big(d \pi(Z)^n v \big)} \\
& \leq \sum_{n=0} ^{m-1} \frac{1}{n!} \rho_m \big(d \pi(Z)^n v \big)  + \sum_{n=m} ^\infty \frac{1}{n!} |Z|^{n-m} \rho_{n} \big(d \pi(Z)^m v \big) <\infty,
\end{align*}
which gives us the wanted conclusion.

Since $d \pi(X_j)$ are continuous linear operators on $\mathcal{H}_\pi ^\infty$, this implies
\begin{align*}
\lefteqn{E_s \left(e^{d \pi (Z)} \right)}\\
& = \sum_{n=0} ^\infty \frac{s^n}{n!} \left( \sum_{1 \leq j_1 , \cdots , j_n \leq d} \| d \pi (X_{j_1} \cdots X_{j_n}) e^{d \pi (Z)} v \|^2 \right)^{\frac{1}{2}} \\
& = \sum_{n=0} ^\infty \frac{s^n}{n!} \left( \sum_{1 \leq j_1 , \cdots , j_n \leq d} \| \sum_{m=0} ^\infty \frac{1}{m!} d \pi (X_{j_1} \cdots X_{j_n}) d \pi (Z)^m v \|^2 \right)^{\frac{1}{2}} \\
& \leq \sum_{n=0} ^\infty \frac{s^n}{n!} \sum_{m=0} ^\infty \frac{1}{m!} \left( \sum_{1 \leq j_1 , \cdots , j_n \leq d} \|d \pi (X_{j_1} \cdots X_{j_n}) d \pi (Z)^m v \|^2 \right)^{\frac{1}{2}} \\
& \leq \sum_{n,m=0} ^\infty \frac{s^n}{n!}\frac{|Z|^m}{m!} \rho_{n+m} (v) = \sum_{l=0} ^\infty \frac{(s+|Z| )^l}{l!} \rho_l (v) = E_{s+ |Z|} (v) <\infty
\end{align*}
where we have used \eqref{eq-rho-comp} in the last inequality.

By Proposition~\ref{prop-analytic evaluation-domain} (3) we can see that the series
    $$e^{d \pi (Z)} v = \sum_{n=0} ^\infty \frac{1}{n!} d \pi(Z)^n v = \sum_{n=0} ^\infty \frac{1}{n!} \sum_{1 \leq j_1 , \cdots j_n \leq d} z_{j_1} \cdots z_{j_n} [d\pi(X_{j_1}) \cdots d\pi(X_{j_n}) v]$$
is absolutely and uniformly convergent on compact subsets of the open disc $\{ (z_1, \cdots, z_d ) \in \mathbb{C}^d : |z_1|^2 + \cdots + |z_d|^2 < t^2 \}$. Hence, it is holomorphic with respect to $Z$ on this disc.
\end{proof}


\subsection{Fourier algebras and weights on the dual of $G$}\label{subsec-weights}

Recall that the Fourier algebra $A(G)$ for a locally compact group $G$ is given by
$$A(G) = \{u\in C_0(G): u(s) = \la \lambda(s) f, g\ra,\; s\in G,\; f,g\in L^2(G)\}.$$
Since we have
    $$\la \lambda(s) f, g\ra = \bar{g}*\check{f}(s),\;\; f,g\in L^2(G),$$
where $\check{f}(s) = f(s^{-1})$, $s\in G$,    
another description of $A(G)$ follows. Consider the contractive map
\begin{equation}\label{eq-conv-map}
\Phi : L^2(G) \otimes_\pi L^2(G) \to C_0(G),\;\; g\otimes f \mapsto g*\check{f},   
\end{equation}
where we recover $A(G)$ by ${\rm Ran}(\Phi)$ equipped with the quotient Banach space structure coming from $(L^2(G) \otimes_\pi L^2(G))/{\rm Ker}\Phi$.

We record here how the group inversion affects coefficient functions of the left regular representation. For $u(s) = \la \lambda(s) f, \bar{g}\ra = g*\check{f}(s),\; s\in G$, we have
    $$\check{u}(s) = \la \lambda(s^{-1}) f, \bar{g}\ra = \overline{\la \lambda(s) \bar{g}, f\ra} = \la \lambda(s) g, \bar{f}\ra = f*\check{g}(s),\;\; s\in G.$$
Thus, we have
\begin{equation}\label{eq-check-flip}
[\Phi(g\otimes f)]^\lor = \Phi(f\otimes g),\;\;g,f\in L^2(G).    
\end{equation}

The dual space of $A(G)$ is well understood. Indeed, we have
    $$A(G)^*\cong VN(G),$$
the group von Neumann algebra acting on $L^2(G)$, with respect to the following duality:
    $$(T, \bar{g}*\check{f}) = \la Tf, g\ra,\;\; T\in VN(G), f,g\in L^2(G).$$
The group structure of $G$ is encoded in $VN(G)$ via the following map
    $$\Gamma : VN(G) \to VN(G\times G)\cong VN(G)\overline{\otimes}VN(G),\; \lambda(x) \mapsto \lambda(x) \otimes \lambda(x),$$
which we call the {\em co-multiplication} on $VN(G)$.

For an abelian group $G$ we have $A(G)\cong L^1(\widehat{G})$ and $VN(G)\cong L^\infty(\widehat{G})$ so that a bounded below weight function $w:\widehat{G} \to (0,\infty)$ has a reciprocal $w^{-1}$ belonging to $L^\infty(\widehat{G})$. This motivates the following definition of ``weight inverse", which is a part of the definition given in \cite[Section 2]{GT}.

\begin{defn}\label{def-weight}
A positive element $\om\in VN(G)$ satisfying ${\rm Ker}\,\om = \{0\}$ and
    $$\om^2 \otimes \om^2 \le \Gamma(\om^2)$$
is called a {\bf weight inverse}.
We call the inverse $W = \om^{-1}$, which is an unbounded positive operator acting on $L^2(G)$ and affiliated to $VN(G)$, a {\bf weight on the dual of $G$}.
\end{defn}

\begin{rem}
Note that Definition \ref{def-weight} focuses only on the case that $\om$ is positive with ${\rm Ker}\,\om = \{0\}$ compared to the definition of ``weight inverse" in \cite[Section 2]{GT}. Moreover, the associated inverse $W$ is ``bounded below" in the sense of \cite[Section 3.2.1]{GLLST}, which means that the case of weights $W$ with unbounded inverse is not covered in this paper nor in \cite{GT}. Note also that conditions (b) and (c) of \cite[Definition 3.3]{GLLST} are automatically satisfied by \cite[Lemma 2.1]{GT}.
\end{rem}

There is a canonical way to transfer a weight on the dual of a closed abelian subgroup $H$ of $G$ to a weight on the dual of $G$. Let
    $$\iota: VN(H) \to VN(G),\;\; \lambda_H(x)\mapsto \lambda(x)$$
be the canonical $*$-homomorphic embedding, where $\lambda_H$ is the left regular representations of $H$. A weight $W_H$ on the dual of $H$ in the sense of Definition \ref{def-weight} induces a weight on the dual of $G$ through the embedding $\iota$, i.e. $W = \iota(W_H)$ is a weight on the dual of $G$ by (3) of \cite[Proposition 3.25]{GLLST} or \cite[Proposition 4.14]{GT}. Here, $\iota(W_H)$ is the unbounded operator explained in \cite[Section~2.1]{GLLST}.

When $H$ is, moreover, a connected abelian Lie subgroup of $G$ we know that it is isomorphic to $\Real^j \times \tor^{k-j}$ for some $0\le j\le k$ as a Lie group. Then, the transferred operator $\Fc_H \circ W_H \circ (\Fc_H)^{-1}$, where $\Fc_H: L^2(H) \to L^2(\widehat{H})$ is the $L^2$-Fourier transform on $H$, is affiliated to the abelian von Neumann algebra $L^\infty(\widehat{H})$ so that there is an essentially bounded below, Borel measurable function
$w: \widehat{H} \to (0,\infty)$
such that
    $$\Fc_H \circ W_H \circ (\Fc_H)^{-1} = M_w,$$
where $M_w$ is the multiplication operator with respect to $w$ acting on $L^2(\widehat{H})$.
Since $W_H$ is a weight on the dual of $H$ we can see that $w$ is a weight function on $\widehat{H}\cong \Real^j \times \z^{k-j}$, which means that it is sub-multiplicative. From \cite[Proposition 3.26]{GLLST} we can find $\beta>0$ and $C>0$ such that
\begin{equation}\label{eq-domination}
    w \le C\cdot w_\beta,
\end{equation}
where $w_\beta$ is the weight function on $\widehat{H}\cong \Real^j \times \z^{k-j}$ given by
\begin{equation}\label{eq-w-beta}
w_\beta(x_1,\cdots, x_j, n_{j+1},\cdots, n_k) := e^{\beta (|x_1| + \cdots + |x_j| + |n_{j+1}| + \cdots + |n_k|)}.
\end{equation}
Recall \cite[Proposition 3.26]{GLLST} to see that
\[
W=w(i\partial\lambda(X_1), \cdots, i\partial\lambda(X_n)),
\]
where the RHS is given by functional calculus. 

For $\beta > 0$ the exponential weight $W_\beta$ on the dual of $G$ of order $\beta$ is defined to be the case where we use $w = w_\beta$.
Then, the above observation tells us that we have
\begin{equation}\label{eq-exp-weight-beta}
    W_\beta := \iota((\Fc_H)^{-1} \circ M_{w_\beta} \circ \Fc_H) =  e^{\beta(|\partial \lambda(X_1)| + \cdots + |\partial \lambda(X_k)|)}.
\end{equation}


\subsection{Beurling-Fourier algebras}\label{subsec-BF-alg}

Let $W$ be a weight on the dual of $G$. We define the weighted Fourier algebra as follows.
\begin{defn}
We define the space
    $$A(G,W) := W^{-1}A(G) = \{W^{-1}\phi: \phi\in A(G)\}\subseteq A(G)$$
equipped with the norm
    $$\|W^{-1}\phi\|_{A(G,W)}:= \|\phi\|_{A(G)}.$$
Here, $W^{-1}\phi$ is the element in $A(G)$ given by
    $$(T, W^{-1}\phi) := (TW^{-1}, \phi),\;\; T\in VN(G).$$
\end{defn}
We know that $A(G,W)$ is a Banach algebra with respect to point-wise multiplication and $A(G,W)^*\cong VN(G)$ with the duality (see \cite[Definition~3.11]{GLLST} or \cite[Proposition 2.3]{GT})
\begin{equation}\label{eq-duality-BF-alg}
(T, W^{-1}\phi)_W := (T,\phi),\;\; \phi\in A(G), T\in VN(G).
\end{equation}

\begin{rem}\label{rem-weighted-Fourier-element}
When $\phi = g*\check{f}$, $g,f\in L^2(G)$ we can immediately check that
    $$W^{-1}\phi = g*(W^{-1}f)^\lor.$$
Thus, for $\xi\in \D(W)$ and $\eta \in L^2(G)$ we have $\eta*\check{\xi} = W^{-1}(\eta*(W\xi)^\lor)$ and
    $$\| \eta*\check{\xi} \|_{A(G,W)} = \|\eta*(W\xi)^\lor\|_{A(G)} \leq \|\eta\|_2 \|W \xi \|_2.$$
\end{rem}

\subsection{Abstract complexification}\label{subsec-abs-complexification}

Let $G$ be a locally compact group. The co-multiplication
    $\Gamma: VN(G) \to VN(G\times G),\; \lambda(s) \mapsto \lambda(s)\otimes \lambda(s)$
can be implemented by the multiplicative unitary $U\in B(L^2(G\times G))$ by
\begin{equation}\label{eq-co-multi}
\Gamma(T) = U^*(I\otimes T)U,
\end{equation}
where $U\xi(s,t) := \xi(ts,t)$, $s,t\in G$, $\xi\in L^2(G\times G)$. The expression \eqref{eq-co-multi} allows us to extend $\Gamma$ to $\overline{VN(G)}$, the space of all closed densely defined operators affiliated to $VN(G)$, via the same formula.

Following the universal complexification model for compact, connected Lie groups by Chevalley (\cite[III. 8]{BtD}), McKennon suggested a model of complexification $G_\Comp$ for a general locally compact group $G$ in \cite{McK2}, where he used the group $W^*$-algebra $W^*(G)$. Recently, Giselsson and Turowska in \cite{GT} developed a modified version of McKennon's model for $G_\Comp$ replacing $W^*(G)$ with $VN(G)$ as follows.

\begin{defn}
The $\lambda$-complexification $G_{\Comp, \lambda}$ of $G$ is defined in \cite{GT} by
    $$G_{\Comp, \lambda} :=\{ T \in \overline{VN(G)}: \Gamma(T) = T\otimes T,\; T\neq 0\}.$$
\end{defn}

While the above definition applies to general locally compact groups, we have a ``{\em Cartan decomposition}" of the complexification for connected Lie groups. The following is \cite[Proposition 3.4]{GT}.

\begin{prop}
Let $G$ be a connected Lie group with the associated Lie algebra $\g$. Then, we have
    $$G_{\Comp, \lambda} =\{ \lambda(s)e^{i\partial \lambda(X)}: s\in G,\; X\in \g\}.$$
\end{prop}

In the cases we focus on in this paper we know that ${\rm Spec}A(G,W)$ can be realized as elements of $G_{\Comp, \lambda}$ in the following sense.

\begin{thm}\label{thm-spectrum-extended-weight-general}
Let $G$ be a connected Lie group and $H$ be a connected abelian Lie subgroup of $G$. Let $W_H$ be a weight on the dual of $H$ and $W = \iota(W_H)$ be the extended weight on the dual of $G$. Then, we have
\begin{align*}
{\rm Spec}A(G,W)
& = \{ \lambda(s)e^{i\partial \lambda_G(X)}W^{-1} \in VN(G):\\
& \quad \quad s\in G,\, X\in \mathfrak{h},\, e^{i\partial \lambda_H(X)}W^{-1}_H \in {\rm Spec}A(H,W_H) \}.
\end{align*}
Furthermore, for $W=W_\beta$ from \eqref{eq-exp-weight-beta} we have
$${\rm Spec}A(G,W_\beta) = \{\lambda(s)e^{i\partial\lambda(X)}W^{-1}_\beta \in VN(G): s\in G, X\in \mathfrak{h}, |X|_\infty \le \beta\},$$
where $|X|_\infty = \max_{j=1} ^d |x_j|$\, for $X = x_1 X_1 + \cdots + x_d X_d \in \g$.
\end{thm}
\begin{proof}
The first assertion is \cite[Theorem 4.23]{GT}. Note that a careful examination of the proof tells that the assumption of simple connectedness of $G$ in \cite[Theorem 4.23]{GT} is redundant. The second assertion boils down to the problem of determining ${\rm Spec}L^1(\widehat{H},w_\beta)$. Then, we only need to find $X = x_1 X_1 + \cdots + x_k X_k \in \mathfrak{h}$ such that there is $C>0$ such that
$$e^{t_1x_1 + \cdots + t_jx_j + n_{j+1}x_{j+1} + \cdots + n_kx_k}\le C\cdot e^{\beta (|t_1| + \cdots + |t_j| + |n_{j+1}| + \cdots + |n_k|)}$$
for any $(t_1, \cdots, t_j, n_{j+1},\cdots, n_k) \in \Real^j \times \z^{k-j}\cong \widehat{H}$. Thus, we end up with the condition $|X|_\infty \le \beta$.
\end{proof}

\begin{rem}\label{rem-spec-duality}
\begin{enumerate}
\item In Theorem \ref{thm-spectrum-extended-weight-general} the condition $$e^{i\partial \lambda_H(X)}W^{-1}_H \in {\rm Spec}A(H,W_H)$$
implies that $e^{i\partial \lambda(X)}W^{-1} \in {\rm Spec}A(G,W)$, and consequently we have $$\D(W)\subseteq \D(e^{i\partial \lambda_G(X)}).$$

\item Elements in ${\rm Spec}A(G,W)$ can be identified as elements of $G_{\Comp, \lambda}$ via the injective map
        $$\lambda(s)e^{i\partial \lambda_G(X)}W^{-1} \mapsto \lambda(s)e^{i\partial \lambda_G(X)}.$$
    The above identification has been used in \cite[Remark 4.8 and Theorem 4.23]{GT}

    \item For $\xi\in \D(W)\subseteq \D(e^{i\partial \lambda_G(X)})$ and $\eta \in L^2(G)$ we have
        $$\eta*\check{\xi} = W^{-1}(\eta*(W\xi)^\lor)$$
    so that the duality \eqref{eq-duality-BF-alg} tells us that
    \begin{align}\label{eq-spec-duality}
    (\lambda(s)e^{i\partial \lambda_G(X)}W^{-1}, \eta*\check{\xi})
    & = (\lambda(s)e^{i\partial \lambda_G(X)}W^{-1}, \eta*(W\xi)^\lor) \\
    & = \la \lambda(s)e^{i\partial \lambda_G(X)}W^{-1} \cdot W\xi, \bar{\eta} \ra \nonumber\\
    & =\la \lambda(s)e^{i\partial \lambda_G(X)}\xi, \bar{\eta} \ra.\nonumber
    \end{align}
\end{enumerate}

\end{rem}

\section{Analytic subalgebras and their density}\label{sec-anal-subalg-density}

Let $G$ be a connected real Lie group
with the associated Lie algebra $\g \cong \Real^d$ as in Section \ref{subsec-Lie} and Section \ref{subsec-weights}.
We fixed a basis $\{X_1, \cdots , X_d\}$ for $\g$ such that $\mathfrak{h} = \Span \{ X_1, \cdots , X_k\}$, $k\le d$, 
is an abelian Lie subalgebra and $H = \exp \mathfrak{h}$ is the connected abelian subgroup of $G$ corresponding to this subalgebra. 
Moreover, the weight $W = \iota(W_H)$ on the dual of $G$ is extended from a weight $W_H$ on the dual of $H$ as in Section \ref{subsec-weights}.

The main goal in this section is to construct a dense subalgebra $\A$ of $A(G,W)$, whose elements can be evaluated at the points of, at least, an open neighborhood of $G$ in $G_\Comp$.
Our candidate is obtained by replacing the space $L^2(G)$ with the space $\Hi^a_R$ in the description of the Fourier algebra $A(G)$ using the map from \eqref{eq-conv-map}. 
More precisely, for $0<R\le \infty$, we consider the map
$$\Phi_R : L^2(G) \otimes_\pi \Hi^a_R \to C_0(G),\;\; g\otimes f \mapsto g*\check{f}.$$
Since the inclusion map $\Hi^a_R \hookrightarrow L^2(G)$ is continuous, we can easily see that $\Phi_R$ is also continuous.

\begin{defn}
For $0<R\le \infty$ we define the space $\A_R$ by the range of the map $\Phi_R$, that is, $\A_R := {\rm Ran}(\Phi_R)$ equipped with the quotient space structure coming from $(L^2(G) \otimes_\pi \Hi^a_R)/{\rm Ker}\,\Phi_R$ as a locally convex space.
\end{defn}

Let us check a nontrivial fact that $\A_R$ is indeed a topological algebra.

\begin{thm}\label{thm-A_R-algebra}
The space $\A_R$ is a Fr\'echet algebra with respect to point-wise multiplication for any $0<R\le \infty$. 
\end{thm}
\begin{proof}
We will follow a standard argument for the fact that a Fourier algebra is a Banach algebra. However, there are additional technicalities involved in dealing with Fr\'echet space structures.

Note that the topology of $\A_R$ is generated by the induced semi-norms
    $$q_s(h) := \inf \{ \sum^k_{i=1}\|g_i\|_2 \cdot E_s(f_i): h = \sum^k_{i=1}g_i*\check{f_i},\, g_i \in L^2(G),\, f_i \in \Hi^a_R \}$$
for $0<s<R$. From the definition of complete projective tensor product of locally convex spaces and the density it is enough to show the following:
for $g_1, g_2 \in C_c(G)$ and $f_1, f_2 \in \Hi^a_R$ we have $\Phi_R (g_1 \otimes f_1) \Phi_R (g_2 \otimes f_2) \in \A_R$ and
\begin{equation*}\label{eq-semi-norm-estimate}
    q_s(\Phi_R(g_1\otimes f_1)\Phi_R(g_2\otimes f_2))\le \|g_1\|_2 \|g_2\|_2\cdot E_{s}(f_1)E_{s}(f_2)
\end{equation*}
for any $0<s<R$. 
Indeed, we can readily check that
    $$(g_1*\check{f}_1)(x)(g_2*\check{f}_2)(x) = \int_G G_y * \check{F}_y(x)\,dy,$$
where $G_y(x) = g_1(xy)g_2(x)$ and $F_y(x) = f_1(xy)f_2(x)$. We will show that the above can be understood as an $\A_R$-valued Pettis integral, i.e. $F_y \in \Hi^a _R$ for a.e. $y \in G$ and
    $$\Phi_R(g_1\otimes f_1)\Phi_R(g_2\otimes f_2) = \int_G G_y * \check{F}_y\,dy.$$
For this goal, it is enough for us to check that the map 
    $$y\mapsto G_y \otimes F_y \in L^2(G) \otimes_\pi \Hi^a_R$$
is well-defined and Pettis integrable, or equivalently, $F_y \in \Hi^a _R$ for a.e. $y \in G$ and the above map is weakly measurable and
    \begin{equation}\label{eq-Pettis-condition}
        \int_G \|G_y\|_2 E_s(F_y)\,dy <\infty\;\; \text{for all}\;\; 0<s<R.
    \end{equation}
Note that we are using \cite[Theorem 3]{Th75}, where the integrability condition is assumed for any continuous semi-norm. However, a quick examination of the proof indicates that the integrability conditions for the semi-norms generating the topology of the underlying Fr\'echet space are enough for the same conclusion.

For weak measurability, we first note that the map
    $$y\mapsto G_y \otimes F_y \in L^2(G\times G)$$
is supported on a compact subset $K\subseteq G$, which allows us to focus on the map $y\in K \mapsto F_y\in L^2(G)$, which is defined almost everywhere and weakly measurable (as a $L^2(G)$-valued function) 
since
    \begin{equation}\label{eq-L2-estimate}
        \int_K \|F_y\|^2_2\,dy = \int_K\int_G|f_1(xy)f_2(y)|^2\,dxdy \le \|f_1\|^2_2\|f_2\|^2_2 < \infty.
    \end{equation}
Let us estimate $\int_G E_s(F_y)^2\,dy \ge \int_K E_s(F_y)^2\,dy$ for a fixed $0<s<R$. Note
    $$d\lambda(X_{j_1}\cdots X_{j_n})F_y = [d\lambda(X_{j_1}\cdots X_{j_n})f_1]^y\cdot f_2 + \cdots + f_1^y \cdot d\lambda(X_{j_1}\cdots X_{j_n})f_2,$$
a sum with $2^n$ terms with mixed Lie derivatives. Here, $h^y$ refers to the right translation with respect to $y$ given by $h^y(x) = h(xy)$, and we used the fact that Lie derivatives commute with right translations.
A computation similar to \eqref{eq-L2-estimate} gives us
\begin{align*}
    \lefteqn{\left(\int_G \sum_{1\le j_1,\cdots,j_n\le d} \int_G |d\lambda(X_{j_1}\cdots X_{j_n})F_y(x)|^2\,dxdy\right)^{\frac{1}{2}}}\\
    & \le \left(\int_G \sum_{1\le j_1,\cdots,j_n\le d} \int_G |[d\lambda(X_{j_1}\cdots X_{j_n})f_1]^y(x)\cdot f_2(x)|^2\,dxdy\right)^{\frac{1}{2}}\\
    & \;\; + \cdots + \left(\int_G \sum_{1\le j_1,\cdots,j_n\le d} \int_G |f_1^y(x) \cdot d\lambda(X_{j_1}\cdots X_{j_n})f_2(x)|^2\,dxdy\right)^{\frac{1}{2}}.
\end{align*}
The first term becomes
\begin{align*}
    \lefteqn{\int_G \sum_{1\le j_1,\cdots,j_n\le d} \int_G |[d\lambda(X_{j_1}\cdots X_{j_n})f_1]^y(x)\cdot f_2(x)|^2\,dxdy}\\
    & = \int_G \sum_{1\le j_1,\cdots,j_n\le d} \int_G |d\lambda(X_{j_1}\cdots X_{j_n})f_1(xy)\cdot f_2(x)|^2\,dxdy\\
    & = \int_G \sum_{1\le j_1,\cdots,j_n\le d} \int_G |d\lambda(X_{j_1}\cdots X_{j_n})f_1(xy)|^2\,dy \cdot |f_2(x)|^2\,dx\\
    & = \sum_{1\le j_1,\cdots,j_n\le d} \int_G |d\lambda(X_{j_1}\cdots X_{j_n})f_1(y)|^2\,dy \cdot \int_G|f_2(x)|^2\,dx\\
    & = \rho_n(f_1)^2\rho_0(f_2)^2.
\end{align*}
The rest of the terms have similar expressions so that we have
    $$\left(\int_G\rho _n(F_y)^2\,dy\right)^{\frac{1}{2}} \le \sum^n_{k=0}\binom{n}{k}\rho_{n-k}(f_1)\rho_k(f_2).
$$
Then, by Cauchy-Schwarz and the above we have
\begin{align}\label{eq-E_s-estimate}
    \lefteqn{\int_G E_s(F_y)^2\,dy}\\
    & = \sum_{m,n\ge 0}\int_G\rho _m(F_y)\rho _n(F_y)\,dy\frac{s^{m+n}}{m!n!}\nonumber\\
    &\le \sum_{m,n\ge 0} \sum^n_{k=0}\sum^m_{l=0}\binom{n}{k}\binom{m}{l}\rho _{n-k}(f_1)\rho _k(f_2)\rho _{m-l}(f_1)\rho _l(f_2)\frac{s^{m+n}}{m!n!}\nonumber\\
    & =\sum_{k,l\ge 0}  \frac{s^k\rho _k(f_1)\rho _l(f_1)s^l}{k!l!}  \left(\sum_{n\ge k} \sum_{m\ge l} \frac{s^{n-k}\rho _{n-k}(f_2)\rho _{m-l}(f_2)s^{m-l}}{(n-k)!(m-l)!}\right)\nonumber\\
    & =\sum_{k,l\ge 0} \frac{s^k\rho _k(f_1)\rho _l(f_1)s^l}{k!l!}  \left(\sum_{n',m'\ge 0}  \frac{s^{n'}\rho _{n'}(f_2)\rho _{m'}(f_2)s^{m'}}{(n')!(m')!}\right)\nonumber\\
    & = E_s(f_1)^2E_s(f_2)^2 <\infty. \nonumber
\end{align}

Now we know that $\int_K E_s(F_y)^2\,dy <\infty$ so that the function $y\in K \mapsto F_y$ belongs to the vector-valued $L^2$-space $L^2(K,X_s)$, where $X_s$ is the Banach space given by
    $$X_s :=\{f\in \Hi^\infty: E_s(f)<\infty\}.$$
Thus, the map $y\in K \mapsto F_y$ is an almost everywhere defined $X_s$-valued function, and consequently an almost everywhere defined $\Hi^a_R$-valued function by considering countably many $0<s<R$ approaching $R$.
Since we already know that the map $y\in K \mapsto (F_y,H)$ is measurable for any $H\in L^2(G)$ and $L^2(G)$ is clearly a total set in $(\Hi^a _R)^*$ we may appeal to \cite[Theorem 1]{Th75} for the weak measurability of the map $y\in K \mapsto F_y\in \Hi^a _R$ and, consequently, of the map $y\in K \mapsto G_y \otimes F_y \in L^2(G) \otimes_\pi \Hi^a_R$.

The condition \eqref{eq-Pettis-condition} comes from
    $$(\int_G \|G_y\|_2 E_s(F_y)\,dy)^2 \le \int_G \|G_y\|^2_2\,dy \cdot \int_G E_s(F_y)^2\,dy$$
and $$\int_G \|G_y\|^2_2\,dy = \int_G\int_G|g_1(xy)g_2(y)|^2\,dxdy \le \|g_1\|^2_2\|g_2\|^2_2
$$
together with \eqref{eq-E_s-estimate}.
\end{proof}


The remainder of this section is devoted to showing the density of the analytic subalgebra $\A_R$ in $A(G,W)$ for certain parameters $R>0$.
The conclusion for a general weight $W$ can be obtained from the same result for the specific weight
    $$W_\beta =  e^{\beta(|\partial \lambda(X_1)| + \cdots + |\partial \lambda(X_k)|)}$$
from \eqref{eq-exp-weight-beta}, whose precise description allows us to perform a thorough investigation. Let us begin with some preparations.

\begin{prop}\label{prop-analytic}
For $r>0$ and $\displaystyle 0< \beta < \frac{r}{\sqrt{k}}$ we have $\Hi^a_r\subseteq \D(W_\beta)$ with
\begin{equation}\label{eq-W-E-norm-compare}
    \|W_\beta v\|\le E_r(v),\;\; v\in \Hi^a_r.
\end{equation}
Moreover, for $v\in \Hi^a_r$ and $w\in \Hi$ the function
    $$t_j\mapsto  \la e^{t_1\partial \lambda(X_1) + \cdots + t_k\partial \lambda(X_k)}W_\beta v, w\ra$$
is real-analytic on $\Real$ for all $1\le j\le k$.
\end{prop}
\begin{proof}
Let $v \in \Hi^a _r$. Let us first check that $v\in \D(W_\beta)$.
\begin{align*}
    \lefteqn{\|[\beta(|\partial \lambda(X_1)| + \cdots + |\partial \lambda(X_k)|)]^nv\|} \\
    & \le \beta^n \sum_{1\le j_1, \cdots, j_n \le k} \||\partial \lambda(X_{j_1}\cdots X_{j_n})|v\|\\
    & = \beta^n \sum_{1\le j_1, \cdots, j_n \le k} \|\partial \lambda(X_{j_1}\cdots X_{j_n})v\|
    \le \beta^n k^{\frac{n}{2}} \rho_n(v) < \infty.
    \end{align*}
Now we appeal to Lemma~\ref{lem-abs-conv} to see that $v\in \D(W_\beta)$ and
\begin{equation}\label{eq-W-beta-series}
    W_\beta v = \sum_{n\ge 0} \frac{[\beta(|\partial \lambda(X_1)| + \cdots + |\partial \lambda(X_k)|)]^nv}{n!}.
\end{equation}
The estimate \eqref{eq-W-E-norm-compare} is now immediate.

For the second statement, we note that we may focus on the case $j=1$ by symmetry. For the unitary
    $$U = e^{t_2\partial \lambda(X_2) + \cdots + t_k\partial \lambda(X_k)}$$
we have $U e^{t_1\partial \lambda(X_1)} = e^{t_1\partial \lambda(X_1) + \cdots + t_k\partial \lambda(X_k)}$ so that
\begin{align*}
    \la e^{t_1\partial \lambda(X_1) + \cdots + t_k\partial \lambda(X_k)}W_\beta v, w\ra
    & = \la U e^{t_1\partial \lambda(X_1)}W_\beta v, w\ra\\
    & = \la e^{t_1\partial \lambda(X_1)}W_\beta v, U^*w\ra.
\end{align*}
Let $|t_1| < r - \sqrt{k} \beta$. By \eqref{eq-W-beta-series} and the Cauchy-Schwartz inequality, we have
\begin{align*}
    \lefteqn{\| \big|t_1\partial \lambda(X_1)\big|^m W_\beta v\| }\\
    & \le |t_1|^m\sum_{n\ge 0} \frac{  \| [\beta(|\partial \lambda(X_1)| + \cdots + |\partial \lambda(X_k)|)]^n \cdot \big|\partial \lambda(X_1)\big|^m v\|}{n!} \\
    &\le \sum_{n\ge 0}\frac{|t_1|^m \beta^n k^{\frac{n}{2}} \rho_{m+n}(v)}{n!}.
\end{align*}
Thus, we have
    $$\sum_{m\ge 0}\frac{\| \big|t_1\partial \lambda(X_1)\big|^m W_\beta v\|}{m!}\le \sum_{l\ge 0}\frac{(|t_1|+\sqrt{k}\beta)^l}{l!}\rho_l(v) < \infty.$$
So, Lemma \ref{lem-abs-conv} tells us that the function
    $$f(t_1) = \la e^{t_1\partial \lambda(X_1)}W_\beta v, U^* w\ra$$
is a power series of $t_1$ convergent on $(-r+\sqrt{k}\beta, r-\sqrt{k}\beta)$.
Analyticity on $\Real$ can be easily obtained by translation in $t_1$ variable, which is the same as taking an appropriate unitary $e^{s\partial\lambda(X_1)}$ ($s\in \Real$) out from the exponential operator.
\end{proof}

\begin{thm}\label{thm-density-domain}
Let $0<R\le \infty$ be such that $\Hi^a_R$ is dense in $L^2(G)$. For any $0<r<R$ and $0<\beta < \frac{r}{\sqrt{k}}$ the space $W_\beta(\Hi^a_r)$ is dense in $L^2(G)$. When $R=\infty$, the same conclusion holds for $r=\infty$.
\end{thm}
\begin{proof}
We only consider the case $R<\infty$ since the case $R=\infty$ can be done similarly.
We take $w\in [W_\beta(\Hi^a_r)]^\perp$ and will show that $w=0$. Using Proposition \ref{prop-holomorphic-ext} we may choose $\eps>0$ so that
    $$e^{t_1\partial \lambda(X_1) + \cdots + t_k\partial \lambda(X_k)}(\Hi^a_R) \subseteq \Hi^a_r$$
for $(t_j)^k_{j=1} \subseteq \Real$ with $\sum^k_{j=1}|t_j|^2 < \eps^2$. Thus, for any $v\in \Hi^a_R$ and $(t_j)^k_{j=1} \subseteq \Real$ with $\sum^k_{j=1}|t_j|^2 < \eps^2$ we have
    $$\la e^{t_1\partial \lambda(X_1) + \cdots + t_k\partial \lambda(X_k)}W_\beta v, w \ra = \la W_\beta e^{t_1\partial \lambda(X_1) + \cdots + t_k\partial \lambda(X_k)}v, w\ra = 0.$$
Note that the operators $e^{t_1\partial \lambda(X_1) + \cdots + t_k\partial \lambda(X_k)}$ and $W_\beta$ come from the functional calculus of the strongly commuting operators $\{ \partial \lambda(X_1) , \cdots , \partial \lambda(X_k) \}$. Thus, the fact $ v \in \D(W_\beta) \cap \D (W_\beta e^{t_1\partial \lambda(X_1) + \cdots + t_k\partial \lambda(X_k)})$ implies the first equality above.

Now we use Proposition \ref{prop-analytic} and the uniqueness of real-analytic functions for $j=1, \cdots, k$ recursively to conclude that
    $$\la W_\beta v, e^{-(t_1\partial \lambda(X_1) + \cdots + t_k\partial \lambda(X_k))}w \ra = \la e^{ t_1\partial \lambda(X_1) + \cdots + t_k\partial \lambda(X_k)}W_\beta v, w \ra = 0$$
for any $t_1, \cdots, t_k \in \Real$. Taking integral over $t_1, \cdots, t_k$ with respect to gaussian measures we get $\la W_\beta v, \xi_\gamma \ra = 0$, where 
    $$\displaystyle \xi_\gamma = \left(\frac{\gamma}{\pi}\right)^{\frac{k}{2}} \int_{\Real^k} e^{-\gamma (t^2_1 + \cdots + t^2_k)} e^{t_1\partial \lambda(X_1) + \cdots + t_k\partial \lambda(X_k)}w\, dt_1\cdots\, dt_k,\;\;\gamma > 0.$$
Lemma \ref{lem-integral-domain} below tells us that $\xi_\gamma \in \D(W_\beta)$ so that $\la v, W_\beta\xi_\gamma \ra = 0$ for any $v\in \Hi^a_R$ and $\gamma>0$. The density of $\Hi^a_R$ gives us the conclusion $W_\beta\xi_\gamma = 0$ and consequently $\xi_\gamma = 0$ since $W_\beta$ is invertible. Finally, since $\gamma > 0$ was arbitrary, a simple application of LDCT shows
    $$w = \lim_{\gamma \to \infty}\xi_\gamma = 0.$$
\end{proof}

\begin{lem}\label{lem-integral-domain}
For any $\beta, \gamma > 0$ and $w\in L^2(G)$ the vector
    $$\xi = \int_{\Real^k} e^{-\gamma\sum^k_{j=1}t^2_j} e^{\sum^k_{j=1}t_j\partial \lambda(X_j)}w\, dt_1\cdots\, dt_k$$
belongs to the domain $\D(W_\beta)$.
\end{lem}
\begin{proof}
For any choice of signs $(\eps_j)^k_{j=1} \subseteq \{\pm 1\}$ we have
    $$\xi = \int_{\Real^k} e^{-\gamma\sum^k_{j=1}t^2_j} e^{\sum^k_{j=1}t_j\eps_j\partial \lambda(X_j)}w\, dt_1\cdots\, dt_k$$
by a simple change of variables. For $t\in \Real$ we consider
    $$\xi(t) = \int_{\Real^k} e^{-\gamma \sum^k_{j=1}(t_j-t)^2} e^{\sum^k_{j=1}t_j\eps_j\partial \lambda(X_j)}w\, dt_1\cdots\, dt_k = H^{it}\xi,$$
where $H = e^{-i\sum^k_{j=1}\eps_j\partial \lambda(X_j)}$, a non-singular self-adjoint positive operator. It is straightforward to see that $\xi(t)$ extends to an entire function and the extension is uniformly bounded on the strip $\{z\in \Comp: |{\rm Im}z|\le \beta\}$.
Now we can appeal to \cite[VI. Lemma 2.3]{Takesaki2} to get $\xi \in \D(H^\beta) = \D(e^{-i\beta\sum^k_{j=1}\eps_j\partial \lambda(X_j)})$.
Since $e^{2\beta\sum^k_{j=1}|t_j|} \le \sum e^{2\beta\sum^k_{j=1}\eps_j t_j}$, $(t_1,\cdots, t_k)\in \Real^k$, where the sum on the RHS is over all the choices of signs $(\eps_j)^k_{j=1} \subseteq \{\pm 1\}$, we get the desired conclusion by functional calculus.
\end{proof}

Now we have the desired density for the weight $W_\beta$.

\begin{cor}\label{cor-density-algebra}
Suppose $0<R\le \infty$ is such that $\Hi^a_R$ is dense in $\Hi$ and take any $0< r < R$.
Then, the algebra $\A_r$ continuously embeds in $A(G, W_\beta)$ with a dense image for any $\displaystyle 0 \leq \beta < \frac{r}{\sqrt{k}}$.
When $R=\infty$, the same conclusion holds for $r=\infty$.
\end{cor}
\begin{proof}
We only consider the case $R<\infty$ since the case $R=\infty$ can be done similarly.
Choose $0<s<r$ so that $\beta<\frac{s}{\sqrt{k}}$. Let $g \in L^2(G)$ and $f \in \Hi_r^a$. Then,  by Proposition~\ref{prop-analytic}, $f \in \D (W_\beta)$ and by Remark \ref{rem-weighted-Fourier-element} and \eqref{eq-W-E-norm-compare} we have
$$\| \Phi_r ( g \otimes f) \|_{A(G,W_\beta)} \leq \|g\|_2 \|W_\beta f \|_2 \leq \|g\|_2 E_s (f).$$
Thus, $\Phi_r$ is continuous as a map from $L^2(G) \otimes_\pi \Hi_r ^a$ into $A(G,W_\beta)$, which means that $\A_r \subseteq A(G, W_\beta)$ and the inclusion is continuous.

For $h \in A(G,W_\beta)$ we have $h = W_{\beta} ^{-1}\phi$ for some $\phi = g*\check{f}$ with $g,f\in L^2(G)$.
For any $\eps>0$ we apply Theorem \ref{thm-density-domain} to choose $\tilde{f}\in \Hi^a_r$ such that $\|f-W_\beta \tilde{f}\|_2\le \eps.$ 
Then, for $\tilde{h} = g*\check{\tilde{f}} \in \A_r$ we have by Remark~\ref{rem-weighted-Fourier-element} that
\begin{align*}
\|h - \tilde{h}\|_{A(G,W_\beta)}
& = \| g * (W_\beta ^{-1}f - \tilde{f})^\lor \|_{A(G, W_\beta)}\\
& = \| g * (f - W_\beta\tilde{f})^\lor \|_{A(G)}\\
& \le \|g\|_2\cdot\|f-W_\beta \tilde{f}\|_2
\le \eps \|g\|_2.
\end{align*}
This explains the density we wanted.
\end{proof}

The density of $\A_r$ in $A(G,W_\beta)$ can be transferred to the case of general $A(G,W)$ using the following observation.

\begin{prop}\label{prop-weight-comparison}
Let $W_1$ and $W_2$ be strongly commuting weights on the dual of $G$ with $W^{-1}_2\le C \cdot W^{-1}_1$ for some $C>0$. Then the space $A(G, W_2)$ continuously embeds in $A(G, W_1)$ with a dense image.
\end{prop}
\begin{proof}
From the definition of Beurling-Fourier algebras it is clear that the space $A(G, W_2)$ embeds continuously in $A(G, W_1)$. We can establish the wanted density, following the same argument as in the proof of Corollary \ref{cor-density-algebra} if $\D(W_2)$ is dense in $\D(W_1)$ in the graph norm. Indeed, we may use \cite[Corollary 5.28]{Schm12} to see that the operators $W_1$ and $W_2$ have a common core, which leads us to the conclusion we need.
\end{proof}

Now we get condition (A1) in the Introduction for certain weight $W$.

\begin{cor}\label{cor-A1}
Suppose $0<R\le \infty$ is such that $\Hi^a_R$ is dense in $\Hi$ and take $\displaystyle 0 \leq \beta < \frac{R}{\sqrt{k}}$.
Let $W_H$ be a weight on the dual of $H$ associated with the weight function $w: \widehat{H} \to (0,\infty)$ such that $w \le C\cdot w_\beta$ for some $C>0$.
Here, $w_\beta$ is the weight function from \eqref{eq-w-beta}.
Then, for $\beta\sqrt{k}<r<R$,
the algebra $\A_r$ continuously embeds in $A(G, W)$ with a dense image, where $W = \iota(W_H)$ be the extended weight on the dual of $G$.
When $R=\infty$, the same conclusion holds for $r=\infty$.
\end{cor}
\begin{proof}
This is immediate from \eqref{eq-domination}, Corollary \ref{cor-density-algebra} and Proposition \ref{prop-weight-comparison}.
\end{proof}

\section{Analytic evaluations and a local solution}\label{sec-local-sol}

In this section, we would like to show that the analytic subalgebra $\A_r$ (satisfying condition (A1) of the Introduction) satisfies (A3'), a ``local" version of condition (A3) of the Introduction.
In order to do that, we first need to check that functions in the analytic subalgebra $\A_r$ (for small $r>0$) extend holomorphically to an open neighborhood of $G$ within $G_\Comp$. 
Here, $G_\Comp$ is a fixed complexification of $G$ in the sense that it is a complex Lie group with the associated Lie algebra $\g_\Comp\cong \Comp^d$ that contains $G$ as a Lie subgroup. 

For $t>0$ we consider open neighborhoods of the origin in the Lie algebras 
$$B_t := \{ X \in \g_\Comp : |X| < t \}\;\; \text{and}\;\; B_t ^\g := B_t \cap \g$$
and open neighborhoods of the identity in the Lie groups
$$\Om_t := \{\exp X: X\in \g_\Comp,\; |X|<t\}\;\; \text{and}\;\;\Om^G_t := \Om_t \cap G.$$ We also consider an open neighborhood $G_t$ of $G$ within $G_\Comp$ (for small $t>0$) given by
\begin{equation}\label{eq-G_t}
    G_t := \{ s \exp(iY): s \in G, \; Y \in B_t^\g \}.
\end{equation}
The openness of $G_t$ for small $t>0$ requires explanations. Indeed, we can choose $t>0$ such that
\begin{equation}\label{eq-E}
    E:G \times \g \rightarrow G_\Comp,\;\; (s, Y) \mapsto s \exp (iY)
\end{equation}
is a diffeomorphism from $\Om_t^{G} \times B^\g _t$ onto an open subset of $G_\Comp$. Note that we used the fact that $E$ is a smooth map with the derivative
$$\g \times \g \ni (X, Y) \mapsto X+iY \in \g_\Comp$$
at the point $(e,0) \in G \times \g$.
Since $E(s_1 s_2, Y) = s_1E(s_2,Y)$ for all $s_1, s_2 \in G$ and $Y \in \g$ by definition, we see that $E$ is also a diffeomorphism from $s\Om_t^{G} \times B^\g _t$ onto an open subset of $G_\Comp$ for each $s \in G$. Thus, we see that
    $$G_t = E ( G \times B_t ^\g ) = \bigcup_{s \in G} E( s \Om_t^G \times B_t ^\g) $$
is an open subset of $G_\Comp$.

We first check the analytic extendability of $\A_r$, which is nontrivial.

\begin{prop}\label{prop-hol-ext}
There exists $r >0$ such that every function $u\in \A_{r}$ allows a unique holomorphic extension to $G_r$.
\end{prop}
\begin{proof}
We choose $T>0$ such that $\exp:B_T \rightarrow \Om_T$ is a biholomorphism and $E$ from \eqref{eq-E} is a diffeomorphism from $\Om_T^G \times B^\g _T$ onto an open subset of $G_\Comp$.
For small $0 < t < T$ such that $\Om_t \Om_t \subseteq \Om_T$ we claim that $E$ is a diffeomorphism from $G \times B^\g _t$ onto $G_t$.
Since $E$ is a local diffeomorphism on $G \times B^\g _t$ by construction, we only need to show that $E$ is injective on $G \times B^\g _t$. Let $(s_1, Y_1), \; (s_2 ,Y_2) \in G \times B^\g _t$ be such that $E( s_1, Y_1 ) = E(s_2 , Y_2)$. Then, $s_1 \exp (iY_1) = s_2 \exp (iY_2)$ and hence
    $$s_1 ^{-1} s_2 = \exp(iY_1) \exp (- iY_2) \in \Om_t \Om_t \subseteq \Om_T.$$
Since $E$ is a diffeomorphism on $s_1 \Om_T \times B_T ^\g$ as before and $(s_1, Y_1) , (s_2, Y_2) \in s_1 \Om_T \times B_T ^\g$, we see that the condition $E(s_1, Y_1) = E(s_2, Y_2)$ implies that $(s_1 , Y_1) = (s_2 ,Y_2)$. Thus, our claim is proved.

Using the fact that the map
    $$\Psi: \g \times i \g \to G_\Comp,\;\; (X,iY) \mapsto \exp(X) \exp(iY)$$
is a diffeomorphism near $(0,0)$, we can choose a product neighborhood $U_1 \times i U_2$ of $(0,0)$ within $B_t ^\g \times i B_t ^\g$, on which this map is a diffeomorphism.
Since $\Psi(U_1 \times i U_2)\subseteq \Om_T$, we get another diffeomorphism $\Phi = \exp^{-1}|_{\Om_T}\circ \Psi$ satisfying
    $$\exp (X) \exp(iY) = \exp \big( \Phi(X,iY) \big),\;\; (X, iY) \in U_1 \times i U_2.$$
Shrinking $U_1$ and $U_2$ if necessary, we may assume that $\Phi$ is given by a restriction of the Baker-Campbell-Hausdorff (shortly, BCH) map to $U_1 \times i U_2 \subseteq \g_\Comp \times \g_\Comp$ (cf. \cite[Appendix~B.4]{Knapp}) and, moreover, $U_2 = B_r ^\g$ for some $0 < r < t$ so that $\Phi$ is a diffeomorphism from $U_1 \times i B_r ^\g$ onto a neighborhood $U_\Comp$ of $0$ within $B_T\subseteq \g_\Comp$.

Now, we are ready to prove the existence of holomorphic extensions of functions in $\A_r$ with the above choice of $0<r<t$. First, we consider the case $u = \bar{g}* \check{f}$ with $g \in L^2(G)$ and $f \in \Hi^a _r$. Using the inverse of the diffeomorphism $E|_{G \times B_r ^\g}$ onto $G_r$ we get a well-defined extension $u_\Comp$ on $G_r$ of $u$ given by
\begin{equation}\label{eq-hol-ext}
u_\Comp(s \exp (iY)):= \la \lambda(s) e^{i \partial \lambda (Y)} f , g \ra,\;\; s \exp (iY) \in G_r.
\end{equation}
We need to check that $u_\Comp$ is indeed holomorphic, which forces us to choose a suitable holomorphic chart system for $G_r$. Indeed, for $s \exp(iY_0) \in G_r$
we choose $\delta>0$ such that $|Y_0| + 2 \delta < r$ and small neighborhoods $0 \in V_1 \subseteq U_1 \cap B_{\delta} ^\g$ and $Y_0 \in V_2 \subseteq Y_0 + B_\delta ^\g \subseteq B_r ^\g$. From the choice of $V_2$ we can see that the open set $s \exp (V_\Comp) \subseteq G_r$ for $V_\Comp := \Phi (V_1 \times i V_2) \subseteq U_\Comp$ together with the coordinates map
    $$s \exp (V_\Comp) \to \g_\Comp,\;\; s \exp(Z) \mapsto Z$$
gives a holomorphic chart for the point $s \exp(iY_0) \in G_r$. Since $\exp(V_\Comp) = \exp(V_1) \exp(i V_2)$ and $i \overline{V_2} \subseteq B_r $ by our choice of $\delta$, we may shrink $V_1$ further to assume that $V_\Comp \subseteq B_r$.

For $Z = \Phi(X, iY) \in V_\Comp$ with $X\in V_1$ and $Y\in V_2$ we have
\begin{equation}\label{eq-u-Comp}
    u_\Comp (s \exp(Z) ) = u_\Comp(s \exp(X) \exp(i Y) ) = \la \lambda(s) e^{\partial \lambda(X)} e^{i \partial \lambda (Y)} f , g \ra.
\end{equation}
On the other hand, we have
\begin{align}\label{eq-double-series}
e^{\partial \lambda (X)} e^{i \partial \lambda (Y)} f
& = e^{\partial \lambda (X)} \sum_{m=0} ^\infty \frac{1}{m!} (i \partial \lambda (Y))^m f \nonumber \\
& = \sum_{m=0} ^\infty \frac{1}{m!} e^{\partial \lambda (X)} ( i \partial \lambda (Y))^m f \nonumber \\
& = \sum_{m=0} ^\infty \sum_{n=0} ^\infty \frac{1}{m! n!} \partial \lambda (X)^n (i\partial \lambda (Y)) ^m f.
\end{align}
Note that we used Lemma~\ref{lem-abs-conv} and the fact that $e^{\partial \lambda(X)} = \lambda (\exp X)$ is a unitary and $\partial \lambda(Y)^m f \in \Hi^a _r$ for each $m \in \mathbb{N}_0$ from Proposition~\ref{prop-analytic evaluation-domain}.
Furthermore, using \eqref{eq-rho-comp} we have
\begin{align*}
    \sum_{m=0} ^\infty \sum_{n=0} ^\infty \frac{1}{m!n!} \| \partial \lambda(X)^n \partial \lambda(Y)^m f \|
    & \leq \sum_{m=0} ^\infty \sum_{n=0} ^\infty \frac{1}{m!n!} |X|^n |Y|^m \rho_{n+m} (f) \\
    & \leq \sum_{k=0} ^\infty \frac{ (|X|+ |Y|)^k}{k!} \rho_k (f)\\
    & \le E_{|X|+ |Y|}(f)< \infty
\end{align*}
since $|X|+ |Y| < \delta + |Y-Y_0 | + |Y_0| < 2 \delta + |Y_0| < r$ and $f\in \Hi^a_r$.
Therefore, the above double series \eqref{eq-double-series} converges absolutely so that we can rearrange the series into a sum of the homogeneous terms to obtain
    $$e^{\partial \lambda (X)} e^{i \partial \lambda (Y)} f = \sum_{k=0} ^\infty \frac{1}{k!} d \lambda \big( \Phi ( X, i Y ) \big)^{k} f = e^{d \lambda( \Phi(X, iY))} f = e^{d \lambda (Z)} f$$
since the BCH map $\Phi$ is given by a universal algebraic formula. Recall that $|Z|<r$ so that Proposition \ref{prop-holomorphic-ext} and \eqref{eq-u-Comp} tell us that $u_\Comp$ is holomorphic near the point $s \exp(iY_0)$ and consequently on $G_r$.

For the uniqueness of the extension, we consider small enough $\eps>0$ to get $\exp (C_\eps )  \subseteq G_r$ for
    $$C_\eps :=\{ Z = z_1 X_1 + \cdots + z_d X_d \in \g_\Comp : |z_j| < \eps, \; \forall j=1 ,\cdots ,d  \}.$$
The domain $C_\eps$ being a cube in $\Comp^d$ allows us to apply the identity theorem for single variable holomorphic functions iteratively to obtain the uniqueness of the holomorphic extension of $u|_{\exp(C_\eps)\cap G}$ to $\exp(C_\eps)$, which is a non-trivial open set in a connect domain $G_r$. Thus, we can now appeal to the identity theorem for several complex variables to reach the desired conclusion.


Since the point evaluations on $G_r$ are continuous linear functionals on $\text{Hol}(G_r)$, the space of all holomorphic functions on $G_r$, the conclusion for general $u \in \A_r$ follows immediately by density if we check that the inclusion map
     $$\Phi_{r}(L^2(G)\odot \Hi^a_{r}) \hookrightarrow {\rm Hol}(G_r)$$
is continuous with respect to the canonical Fr\'echet space structures on $\A_{r}$ and ${\rm Hol}(G_r)$. Let $0 < r' < r$ and note that we have
\begin{align*}
|\la e^{ i\partial\lambda(Y)} f, \lambda(s^{-1})g\ra|
&\le \|g\|\cdot \|e^{i \partial\lambda(Y)} f\|
\le \|g\|\cdot \sum_{n\ge 0}\frac{  \|(i \partial\lambda(Y))^nf\|}{n!}\\
&\le \|g\|\cdot \sum_{n\ge 0}\frac{|Y|^n\rho_n(f)}{n!}
\le \|g\|\cdot E_{|Y|}(f)
\le \|g\| \cdot E_{r'} (f)
\end{align*}
for any $s\in G$ and $Y \in B^\g _{r'}$, which proves the continuity we wanted.

The uniqueness of the holomorphic extension for this case is also proved in the same manner.
\end{proof}

Now we prove condition A(3'), which is a ``local" version of condition (A3) for a specific choice of weights $W_\beta$.

\begin{thm}\label{thm-eval-action}
Suppose that $r>0$ is the constant chosen in Proposition \ref{prop-hol-ext} and $\displaystyle 0 < \beta < \frac{r}{\sqrt{k}}$. Then, every element in ${\rm Spec}A(G,W_\beta)$ acts on the algebra $\A_r$ as an evaluation functional on the points of an open neighborhood $G_r$ 
of $G$. 
\end{thm}
\begin{proof}
Recall Theorem \ref{thm-spectrum-extended-weight-general} saying that
$${\rm Spec}A(G,W_\beta) = \{\lambda(s)e^{i\partial\lambda(X)}W^{-1}_\beta \in VN(G): s\in G, X\in \mathfrak{h}, |X|_\infty \le \beta\}.$$
First, we consider $u = \bar{g}* \check{f} \in \A_r$ for $g \in L^2(G)$ and $f \in \Hi^a _r$. By Proposition \ref{prop-hol-ext} the function $u$ has a holomorphic extension $u_\Comp$ to $G_r\subseteq G_\Comp$. Since we have $\Hi^a_r\subseteq \D(W_\beta)$ by Proposition \ref{prop-analytic}
we may apply \eqref{eq-spec-duality} and \eqref{eq-hol-ext} to get
    $$(\lambda(s)e^{i\partial\lambda(X)}W^{-1}_\beta, u) = u_\Comp(s\exp(iX)),$$
where $|X|\le \sqrt{k}\cdot |X|_\infty \le \beta \sqrt{k} <r$.

The case of a general element in $\A_r$ can be done by density.
\end{proof}

\begin{rem}
Requiring the constant $r>0$ in Theorem \ref{thm-eval-action} to be $r<R$ for some $R\ge 0$ such that $\Hi_R$ is dense in $\Hi$
we get a subalgebra $\A_r$ satisfying condition (A1) (by Corollary \ref{cor-density-algebra}) and a ``local" version of condition (A3) for the weight $W_\beta$ with $\displaystyle 0 < \beta < \frac{r}{\sqrt{k}}$.
\end{rem}

\section{A global solution for connected, simply connected and nilpotent groups}\label{sec-global-sol}

In this section, we focus on the case of a connected, simply connected and nilpotent Lie group $G$ with the associated Lie algebra $\g \cong \Real^d$. 
Note that the complexification $\g_\Comp$ of $\g$ corresponds to the uniquely determined simply connected complex Lie group $G_\Comp$, which is a complexification of $G$.
We fix an ordered {\em Jordan-H\"{o}lder} basis $\{X_1, \cdots , X_d\}$ for $\g$ such that $\mathfrak{h} = \Span \{ X_1, \cdots , X_k\}$, $k\le d$, is an abelian Lie subalgebra and $H = \exp \mathfrak{h}$ is the connected abelian subgroup of $G$ corresponding to $\mathfrak{h}$.
Furthermore, the weight $W = \iota(W_H)$ on the dual of $G$ extends from a weight $W_H$ on the dual of $H$ as in Section \ref{subsec-weights}. Our goal in this section is to find a subalgebra $\A$ of $A(G,W)$ satisfying conditions (A1), (A2') and (A3") in the Introduction.

An immediate candidate for our subalgebra $\A$ would be the Fr\'echet algebra $\A_\infty$. Since $G_\Comp = \bigcup_{t>0}G_t$, where $G_t$ is from \eqref{eq-G_t}, we get the following by \eqref{eq-domination}, Corollary \ref{cor-A1} and Theorem \ref{thm-eval-action}. 

\begin{cor}
The Fr\'echet algebra $\A_\infty$ is continuously embedded in the Beurling-Fourier algebra $A(G,W)$ with a dense image and every element in ${\rm Spec}A(G,W)$ acts on $\A_\infty$ as an evaluation functional at the points of $G_\Comp$.
\end{cor}

The above theorem says that $\A_\infty$ satisfies condition (A1) and condition (A3) in the Introduction for any choice of weight $W$. 
Unfortunately, we were unable to check any reasonable variant of condition (A2) for $\A_\infty$, which led us to a different subalgebra $\check{\mathcal{C}}$ satisfying conditions (A1), (A2') and (A3") in the Introduction. Note that (A2') and (A3") are ``partial" versions of (A2) and (A3), respectively.

\subsection{A modified analytic subalgebra}

Recall $\rho: G \to B(L^2(G))$, the right regular representation of $G$, given by
$$\rho(s)f(x) = f(xs),\;\; x\in G.$$
Recall also $\lambda|_H$, the left regular representation of $G$ restricted to the subgroup $H$.

\begin{defn}
Let $\widetilde{\Hi}^a_\infty$ be the Fr\'echet space given by
$$\widetilde{\Hi}^a_\infty := \{f\in \Hi^\infty_\rho: d\rho(X_{j_1}\cdots X_{j_n})f \in \Hi^a_{\lambda|_H,\infty}, \,n\in \n,\, 1\le j_1,\cdots,j_n\le d\}$$
equipped with the semi-norms
    $$E^H_{s,j_1,\cdots,j_n}(f) := E^{\lambda|_H}_s(d \rho(X_{j_1}\cdots X_{j_n})f),\;\;f\in \widetilde{\Hi}^a_\infty$$
for all $s>0$, $n\in \n$ and $1\le j_1,\cdots,j_n\le d$, where $E^{\lambda|_H}_s(\cdot)$ is the semi-norm from
\eqref{eq-seminorm-E}.
For the continuous linear map 
$$\Phi^H_\infty : L^2(G)\otimes_\pi \widetilde{\Hi}^a_\infty \to C_0(G),\;\; g\otimes f \mapsto g*\check{f}$$
we define the Fr\'echet space
    $$\B := {\rm Ran}(\Phi^H_\infty)$$
equipped with the quotient locally convex space structure.
\end{defn}

\begin{rem}\label{rem-Laplacian-continuous}

\begin{enumerate}
    \item From the definition it immediately follows that the operator $d\rho(X_{j_1}\cdots X_{j_n})$ is continuous in the space $\widetilde{\Hi}^a_\infty$ for any $n\ge 1$ and $1\le j_1,\cdots,j_n\le d$. 
    \item We do not know whether the space $\B$ is an algebra at the time of this writing.
\end{enumerate}

\end{rem}

We need another subspace of $\B$ as the ``push-forward" of a function algebra on $\g\cong \Real^d$ via the ``canonical coordinates map of the second kind" $\Psi$ given by
\begin{equation}\label{eq-coordinates-map}
\Psi: \Real^d \to G,\;\; (t_1,\cdots, t_d) \mapsto \exp(t_1X_1)\exp(t_2X_2)\cdots \exp(t_dX_d). 
\end{equation}
Recall that for a function $F$ on $\g\cong \Real^d$ its push-forward $\Psi_*(F)$ on $G$ is given by
    $$\Psi_*(F) := F\circ\Psi^{-1}.$$

\begin{rem}\label{rem-Phi-Haar}
The push-forward $\Psi_*(\cdot)$ can be easily extended to the case of measures and it is well-known that $\Psi_*(dm)$ is a left Haar measure on $G$ for the Lebesgue measure $dm$ on $\Real^d$.    
\end{rem}

\begin{defn}\label{def-C-check}
We define the space $\mathcal{C}$ by
    $$\mathcal{C} := \Psi_*[\Fc^{\Real^d}(C^\infty_c(\Real^k) \odot \mathcal{S}(\Real^{d-k}))],$$
where $C^\infty_c(\Real^k)$ and $\mathcal{S}(\Real^{d-k})$ refer to the space of test functions on $\Real^k$ and the Schwartz class on $\Real^{d-k}$, respectively,
and $\Fc^{\Real^d}$ is the Fourier transform on $\Real^d$.
The locally convex topological algebra structure of $\mathcal{C}$ comes from, through the maps $\Fc^{\Real^d}$ and $\Psi_*$, the algebraic tensor product $C^\infty_c(\Real^k) \odot \mathcal{S}(\Real^{d-k})$, which is an algebra with respect to convolution, so that $\mathcal{C}$ is an algebra with respect to pointwise multiplication.

We also consider the space $\check{\mathcal{C}} := \{\check{f}: f\in \mathcal{C}\}$ whose locally convex topological algebra  structure comes from $\mathcal{C}$ through the check map
$$J: \mathcal{C} \to \check{\mathcal{C}},\;\; f\mapsto \check{f},$$
which means that
the map $J$ is an isomorphism of locally convex topological algebras.
\end{defn}

Note that the algebra multiplication of $\check{\mathcal{C}}$ is still pointwise multiplication.
We would like to show that the algebra $\check{\mathcal{C}}$ is indeed a dense subspace in $\B$, which requires a better understanding of the semi-norm structures in $\widetilde{\Hi}^a_\infty$.
Let us begin with the description of the space of entire vectors $\Hi^a_{\pi,\infty}$ for a strongly continuous unitary representation $\pi: G \to B(\Hi_\pi)$
using its restrictions $\pi_j = \pi|_{G_j}$ to the one-parameter subgroups
$$G_j := \exp(\Real \cdot X_j) = \{\exp(tX_j): t\in \Real\},\;\; 1\le j\le d.$$

\begin{prop}\label{prop-entire-domain} Let $G$ be a connected, simply connected and nilpotent Lie group. Let $\pi$ and $\pi_j$ ($1\le j\le d$) be the representations as above.
Then, we have an isomorphism
    $$\Hi^a_{\pi,\infty} \cong \bigcap_{1\le j\le d}\Hi^a_{\pi_j, \infty}$$
as Fr\'{e}chet spaces.
\end{prop}
\begin{proof}
Identification at the set level is immediate from \cite[Theorem 1.1]{Goodman_nilpotent}. For the isomorphism as Fr\'{e}chet spaces we can appeal to the open mapping theorem for Fr\'{e}chet spaces.
\end{proof}

We will use the above for the case $\pi = \lambda|_H$ on $H$. Note that $(\lambda|_H)_j = \lambda_j$ for $1\le j \le k$. For this range of $j$ we can appeal to the fact that $H$ is abelian to conclude that
$$\exp(tX_j)\Psi(t_1,\cdots, t_d) = \Psi(t_1,\cdots, t_j+t,\cdots, t_d),$$
where $\Psi$ is the coordinates map from \eqref{eq-coordinates-map}.
Thus, we have
\begin{equation}\label{eq-U_j-lambda}
    U_j(t)(F) = \lambda(\exp(tX_j))[F\circ \Psi^{-1}]\circ \Psi,\;\; f\in L^2(G),
\end{equation}
where $U_j$ is a unitary representation of $\Real$ acting on the Hilbert space $L^2(\Real^d)$ given by
    $$U_j(t)F(t_1,\cdots, t_d) = F(t_1,\cdots, t_j - t,\cdots, t_d),\;\; F\in L^2(\Real^d).$$
The identity \eqref{eq-U_j-lambda} tells us the following.

\begin{prop}\label{prop-lambda-U}
For $1\le j \le k$ we have an isomorphism
    $$\Hi^a_{U_j, \infty} \ni F \longmapsto F\circ \Psi^{-1} \in  \Hi^a_{\lambda_j, \infty}$$
as Fr\'{e}chet spaces.
\end{prop}

We need an alternative description of the Fr\'{e}chet space structure of $\Hi^a_{U_j, \infty}$.
\begin{prop}\label{prop-U_j-another-description}
For $1\le j \le d$ we have the following.
\begin{enumerate}
    \item A function $F\in L^2(\Real^d)$ belongs to the space $\Hi^a_{U_j, \infty}$ if and only if $F$ extends entirely for the $j$-th variable and
\begin{equation}\label{eq-norm-U_j}
    \sigma_s(F) := \sup_{|\eta|\le s}\left(\int_{\Real^d}|F(t_1,\cdots, t_j+i\eta, \cdots, t_d)|^2\,dt_1\cdots\, dt_d\right)^{\frac{1}{2}}
\end{equation}
are finite for all $s>0$. Furthermore, the family of norms $\{\sigma_s\}_{s>0}$ gives the topology of $\Hi^a_{U_j, \infty}$.

\item The following map is an isomorphism of Fr\'{e}chet spaces.
$$\Hi^a_{U_j, \infty} \to \bigcap_{s>0}L^2(\Real^d; e^{2s|\xi_j|}),\;\; F\mapsto \widehat{F},$$
where $\widehat{F}$ is the Fourier transform of $F$ on $\Real^d$.
\end{enumerate}
\end{prop}
\begin{proof}
(1) We repeat the same argument of \cite[Proposition 5.1, Corollary 5.1]{Goodman_entire}.

\vspace{0.5cm}

(2) We only need to observe the equivalence
$$\sigma_s(F) \cong \|\widehat{F}\|_{L^2(\Real^d; e^{2s|\xi_j|})},\;\; F\in \Hi^a_{U_j, \infty}$$
upto a universal constant by applying the Fourier transform on $\Real^d$.
\end{proof}

We need yet another characterization of the functions in $\Hi^a_{\lambda|_H, \infty}$ using the {\em group Fourier transform}. Recall that $\widehat{G}$ is the unitary dual of $G$ so that for each $\xi \in \widehat{G}$ there is an associated irreducible unitary representation 
$$\pi^\xi: G\to B(\Hi_\xi).$$ For $f\in L^1(G)$ and $\xi \in \widehat{G}$ we recall
    $$\hat{f}(\xi) := \int_G f(x)\pi^\xi(x)\,dx \in B(\Hi_\xi).$$
The Plancherel theorem says that there is a unique Borel measure $\mu$ (called the Plancherel measure) on $\widehat{G}$ such that
$$\|f\|^2_{L^2(G)} = \int_{\widehat{G}}\|\hat{f}(\xi)\|^2_2\,d\mu(\xi),\;\; f\in L^1(G)\cap L^2(G),$$
where $\|\cdot\|_2$ denotes the Hilbert Schmidt norm. This identity allows us to define $\hat{f}(\xi)$ for $f\in L^2(G)$ and $\xi\in \widehat{G}$ by extension.

A strongly continuous unitary representation $\pi: G\to B(\Hi)$ can be extended to a holomorphic representation of $G_\Comp$ as follows. For $Z\in \g_\Comp$ and $v\in \Hi_{\pi, \infty}$ we define
$$\pi_\Comp(\exp(Z))v := e^{d\pi(Z)}v,$$
where the latter is from
\eqref{eq-rep-comp}.

\begin{thm}\label{thm-integral-criterion}
A function $f\in L^2(G)$ is an entire vector for $\lambda|_H$ if and only if the Fourier transform $\hat{f}$ satisfies
\begin{equation*}
{\rm Ran}(\hat{f}(\xi))\subseteq \Hi^a_{\pi^\xi|_H,\infty} \;\;\text{$\mu$-almost every $\xi$ and}
\end{equation*}

\begin{equation}
\int_{\widehat{G}}\sup_{\gamma \in \{\exp(X)\,:\, X\in \mathfrak{h}_\Comp,\, |X|<t\} }\|\pi^\xi_\Comp(\gamma)\hat{f}(\xi)\|^2_2\,d\mu(\xi)<\infty
\end{equation}
for all $t>0$,
where $\mu$ is the Plancherel measure on $\widehat{G}$.
\end{thm}
\begin{proof}
We use \cite[Lemma 3.1 and 3.2]{Goodman_nilpotent} for $\displaystyle \lambda|_H \cong \int^\oplus_{\widehat{G}} \pi^\xi|_H \,d\mu(\xi)$.
\end{proof}

Now we are ready to prove the main technical result in this section.

\begin{thm}\label{thm-B-embeds-into-A}
\begin{enumerate}
\item The algebra $\mathcal{C}$ is continuously embedded in $\widetilde{\Hi}^a_\infty$ with a dense image.

\item The algebra $\check{\mathcal{C}}$ embeds continuously in $\B$ with a dense image.
\end{enumerate}
\end{thm}
\begin{proof}
(1) Applying Proposition \ref{prop-entire-domain}, \ref{prop-lambda-U} and \ref{prop-U_j-another-description} we know that the topology of $\widetilde{\Hi}^a_\infty$ is described by the semi-norms
$$\sigma_{s,j_1,\cdots,j_n}(F\circ\Psi^{-1})
:= \|\Fc^{\Real^d}(d \rho(X_{j_1}\cdots X_{j_n})[F\circ\Psi^{-1}]\circ \Psi)\|_{L^2(\Real^d; e^{2s|\xi_j|})}$$
for all $s>0$, $n\in \n$, $1\le j_1,\cdots,j_n\le d$ and $1\le j\le k$.
Note that (see \cite[Appendix A.2]{CG90}, for example)
$$ \sigma_{s,j_1,\cdots,j_n}(F\circ\Psi^{-1}) = \|D_{j_1,\cdots,j_n}\widehat{F}\|_{L^2(\Real^d; e^{2s|\xi_j|})}$$
for some partial differential operator $D_{j_1,\cdots,j_n}$ on $\Real^d$ with polynomial coefficients. Thus, we can immediately see that $\mathcal{C}$ is continuously embedded in $\widetilde{\Hi}^a_\infty$.
Moreover, a standard test function approximation argument (for example \cite[Theorem 3.3]{HH08}) tells us that $\Fc^{\Real^d}(C^\infty_c(\Real^k) \odot C^\infty_c(\Real^{d-k}))$, and consequently
$\mathcal{C}$, are dense in $\widetilde{\Hi}^a_\infty$.

\vspace{0.5cm}

(2) Let us define the space $[\widetilde{\Hi}^a_\infty]^\lor$ as in the case of $\check{\mathcal{C}}$ in Definition \ref{def-C-check} so that $\check{\mathcal{C}}$ is continuously embedded in $[\widetilde{\Hi}^a_\infty]^\lor$ with a dense image.
For each $f\in [\widetilde{\Hi}^a_\infty]^\lor$ and integer $m>d(G)/4$, where $d(G)$ is the real dimension of the Lie group $G$, we have
    $$f = E_m * (1-\Delta)^m f = E_m * ((1-\Om)^m \check{f})^\lor$$
for some $E_m \in L^2(G)$ by \cite[(3.8) and Lemma 3.3]{LT06}. Here,
$$\Delta = d\lambda(X_1)^2 + \cdots + d\lambda(X_d)^2 \;\;\text{and}\;\; \Om = d\rho(X_1)^2 + \cdots + d\rho(X_d)^2$$
is the right and left invariant Laplacians, respectively, and we used the fact that
$$d\rho(X)\check{f} = [d\lambda(X)f]^\lor,\;\; X\in \g.$$
Remark \ref{rem-Laplacian-continuous} (2) tells us that the operator $(1-\Om)^m$ is continuous on $\widetilde{\Hi}^a_\infty$ so that $[\widetilde{\Hi}^a_\infty]^\lor$ (and consequently $\check{\mathcal{C}}$) continuously embeds in $\B$.

The $L^p(G)$ $(1\le p<\infty)$ norm on the algebra $\mathcal{C}$ coincides with the $L^p(\Real^d)$ norm on the space $\mathcal{C}$ via the map $\Phi$ (see Remark \ref{rem-Phi-Haar}),
which means that $\mathcal{C}\subseteq L^1(G)$.
Note that $f\in \mathcal{C}$ implies $\bar{f}\in \mathcal{C}$ and, hence, the operator norm of $\widehat{\check{f}}(\pi) = [\widehat{\bar{f}}(\pi)]^*$ is uniformly bounded for each irreducible representation $\pi\in \widehat{G}$.
Thus, Theorem \ref{thm-integral-criterion} tells us that $g*\check{f}\in \Hi^a_{\lambda|_H,\infty}$ for $g \in \widetilde{\Hi}^a_\infty$. For any $X\in \g$ we can easily see that
    $$d\rho(X)[g*\check{f}] = g*[d\lambda(X)f]^\lor.$$
Since $\mathcal{C}$ is clearly closed under the operator $d\lambda(X)$ for any $X\in \g$ we can conclude that $d \rho(X_{j_1}\cdots X_{j_n}) (g *\check{f} ) \in \Hi^a_{\lambda|_H,\infty}$ for any $n\in \n$ and $1\le j_1,\cdots,j_n\le d$,
which explains $g*\check{f} \in \widetilde{\Hi}^a_\infty$.
Thus, we have 
$$\Phi^H _\infty (\mathcal{C} \odot \widetilde{\Hi}^a_\infty) = [\Phi^H _\infty (\widetilde{\Hi}^a_\infty\odot \mathcal{C})]^\lor \subseteq [\widetilde{\Hi}^a_\infty]^\lor,$$
which, together with the density of $\mathcal{C}$ in $L^2(G)$, explains the density of $[\widetilde{\Hi}^a_\infty]^\lor$, and consequently, the density of $\check{\mathcal{C}}$ in $\B$.
\end{proof}

We collect more properties of the spaces $\widetilde{\Hi}^a_\infty$, $\B$ and $\mathcal{C}$.

\begin{prop}\label{prop-properties-of-tilde-domain}
\begin{enumerate}
\item For any $\beta>0$ we have 
$$\widetilde{\Hi}^a_\infty \subseteq \Hi^a_{\lambda|_H,\infty}\subseteq \D(W_\beta).$$

\item For any $Z\in \mathfrak{h}$ we have
    $$e^{d\lambda(Z)}\widetilde{\Hi}^a_\infty \subseteq \widetilde{\Hi}^a_\infty.$$

\item For any $\beta>0$ the spaces $\widetilde{\Hi}^a_\infty$ and $W_\beta(\widetilde{\Hi}^a_\infty)$ are dense in $L^2(G)$.

\item The space $\B$ is continuously embedded in $A(G,W)$ with a dense image for any weight $W = \iota(W_H)$ extended from a weight $W_H$ on the dual of $H$ as in Section \ref{subsec-weights}.
\end{enumerate}
\begin{proof}
(1) For the inclusion $\Hi^a_{\lambda|_H,\infty}\subseteq \D(W_\beta)$ the same approach as in Proposition \ref{prop-analytic} works.

\vspace{0.5cm}

(2) We use Proposition \ref{prop-holomorphic-ext} to see that $e^{d\lambda(Z)}\Hi^a_{\lambda|_H,\infty} \subseteq \Hi^a_{\lambda|_H,\infty}$ for $Z\in \mathfrak{h}$. Once we observe that the operators $d\rho(X_{j_1}\cdots X_{j_n})$ and $d\lambda(Z)$ commute for any $n\ge 1$ and $1\le j_1,\cdots,j_n\le d$ we get the desired conclusion.

\vspace{0.5cm}

(3) The density of $\widetilde{\Hi}^a_\infty$ in $L^2(G)$ is clear as it contains $\mathcal{C}$ as we have seen in Theorem \ref{thm-B-embeds-into-A} (1). The second density follows from the same proof as in Theorem \ref{thm-density-domain}.

\vspace{0.5cm}

(4) We can repeat the same argument in the proof of Corollary \ref{cor-density-algebra} for the weight $W_\beta$. For a general weight $W$ we can appeal to \eqref{eq-domination} and Proposition \ref{prop-weight-comparison}.

\end{proof}

\end{prop}

\subsection{A partially holomorphic evaluation}

Since we consider evaluations of certain functions on a closed subset of $G_\Comp$ containing $G$, we need a better understanding of the corresponding domain of extension.

\begin{prop}\label{prop-partial-complexification-nilpotent}
The map
\begin{equation}\label{eq-partial-Cartan-decomposition}
    G \times \mathfrak{h} \ni (s, Y) \longmapsto s \exp(iY) \in G_\Comp
\end{equation}
is a diffeomorphism onto its image
\begin{equation}\label{eq-partial-domain}
    G^{\mathfrak{h}} := \{ s \exp(iY) : s \in G, \; Y \in \h \},
\end{equation}
which is a properly embedded submanifold of $G_\Comp$ (meaning that it is a closed subset of $G_\Comp$, which is also an embedded submanifold).
\end{prop}
\begin{proof}
Composing $\Psi : \g \rightarrow G$ from \eqref{eq-coordinates-map} on the left component, the map \eqref{eq-partial-Cartan-decomposition} becomes
    $$\g \times \mathfrak{h} \ni (X,Y) \longmapsto \Psi(X) \exp(iY) = \exp(i \text{Ad}_{\Psi(X)} Y) \Psi(X) \in G_\Comp.$$
Since the map
    $$\g \times \mathfrak{h} \ni (X,Y) \longmapsto (X, \text{Ad}_{\Psi(X)} Y ) \in \g \times \mathfrak{h}$$
is a well-defined diffeomorphism with the obvious inverse (note that $\h$ is an ideal of $\g$), we only need to show that the map
    $$\g \times \mathfrak{h} \ni (X,Y) \longmapsto \exp(iY) \Psi(X) \in G_\Comp$$
is a diffeomorphism onto a properly embedded submanifold of $G_\Comp$. If we write $X = a_1 X_1 + \cdots + a_d X_d \in \g$ and $Y = b_1 X_1 + \cdots + b_k X_k \in \h$, then, since $\h$ is abelian,
\begin{align*}
    \lefteqn{\exp(iY) \Psi(X) = \exp(i b_1 X_1) \cdots \exp( i b_k X_k) \exp(a_1 X_1 ) \cdots \exp(a_d X_d)} \\
    &\;= \exp \big( (a_1 + i b_1) X_1 \big) \cdots \exp \big( (a_k + i b_k ) X_k \big) \exp(a_{k+1} X_{k+1}) \cdots \exp(a_d X_d ),
\end{align*}
which is the restriction of the diffeomorphism $\Psi_\Comp : \g_\Comp \rightarrow G_\Comp$ given by
    $$\Psi_\Comp(z_1 X_1 + \cdots + z_d X_d ) = \exp(z_1 X_1) \,\cdots \,\exp(z_d X_d ), \quad z_1, \cdots, z_d \in \Comp.$$
to the properly embedded submanifold $\g \oplus i \h \subseteq \g_\Comp$.
Note that the above $\Psi_\Comp$ is a diffeomorphism since $\{X_1, iX_1, X_2, iX_2, \cdots, X_d, iX_d \}$ forms a real Jordan-H\"older basis of $\g_\Comp$ (cf. \cite[Corollary~1.126]{Knapp}). Now the conclusion follows.
\end{proof}

Note that the domain $G^{\mathfrak{h}} = \bigsqcup_{s \in G} s H_\Comp$ from \eqref{eq-partial-domain} is in fact a disjoint union of the left cosets $s H_\Comp$,
each of which carries a canonical complex manifold structure, which can be described as follows.
Recall that the connected subgroup $H_\Comp$ of $G_\Comp$ associated with the Lie subalgebra $\h_\Comp \leq \g_\Comp$ is simply-connected by \cite[Corollary~1.134]{Knapp} and the map
    $$\mathfrak{h}_\Comp \ni Z \mapsto \exp(Z) \in H_\Comp$$
is a biholomorphism.
For each $s \in G$, the left coset $sH_\Comp$ is a properly embedded submanifold of $G_\Comp$ diffeomorphic to $H_\Comp$
via the coordinates map
    $$\h_\Comp \ni Z \longmapsto s\exp(Z) \in sH_\Comp.$$

Now we consider evaluations of the functions in $\B$ on a certain subset of $G_\Comp$ containing $G$, which requires an appropriate extension procedure. Note that the corresponding domain of extension is
$$G^{\mathfrak{h}} = \bigsqcup_{s \in G} s H_\Comp$$
from \eqref{eq-partial-domain}.

\begin{defn}
We say that a function  $v : G^\h \rightarrow \Comp$ is {\bf $H$-holomorphic} if $v$ is continuous on $G^\h$ and is holomorphic when restricted to $sH_\Comp$ for each $s \in G$.
\end{defn}

\begin{thm}\label{thm-partial-hol-ext}
Each function $u \in \B$ admits a unique $H$-holomorphic extension to $G^{\mathfrak{h}}$.

\end{thm}

\color{black}

\begin{proof}
First, we consider the case $u = \bar{g}* \check{f}$ with $g\in L^2(G)$ and $f \in \widetilde{\Hi}^a_\infty\subseteq \Hi^a _{\lambda|_H , \infty}$. Define the extension $u^{\mathfrak{h}}_\Comp$ of $u$ to $G^{\mathfrak{h}}$ by
\begin{equation}\label{eq-hol-ext-nilpotent}
    u^{\mathfrak{h}}_\Comp (s \exp(iY) ) := \la \lambda(s) e^{i \partial \lambda(Y)} f , g \ra, \;\; s \in G, \; Y \in \mathfrak{h},
\end{equation}
which is well-defined thanks to Proposition~\ref{prop-partial-complexification-nilpotent}.

Toward the holomorphicity of $u^{\mathfrak{h}}_\Comp$ on each coset $sH_\Comp$, $s\in G$, we consider $Z = X+ iY \in \mathfrak{h}_\Comp$. Since $X$ and $Y$ commute, we have $\exp(Z) = \exp(X) \exp(iY)$ and $\displaystyle e^{d \lambda(Z)} f = e^{d \lambda(X)} e^{i d \lambda(Y)} f$.
Thus, the map
\begin{align*}
\mathfrak{h}_\Comp \ni Z \mapsto u^{\mathfrak{h}}_\Comp \big(s \exp(Z) \big) 
& = u^{\mathfrak{h}}_\Comp \big(s \exp(X) \exp(iY) \big)\\
& = \la \lambda(s) e^{d \lambda(X)} e^{i d \lambda(Y)} f , g \ra\\
& = \la \lambda(s) e^{d \lambda(Z)} f , g \ra
\end{align*}
is holomorphic thanks to the choice $f\in \Hi^a _{\lambda|_H , \infty}$.

For the continuity of $u^{\mathfrak{h}}_\Comp$ we consider sequences $(s_j)_j \subseteq G$ and $(Y_j)_j \subseteq \h$ such that $s_j \rightarrow s$ in $G$ and $Y_j \rightarrow Y$ in $\h$. 
Since $u^{\mathfrak{h}}_\Comp$ is holomorphic on each coset $sH_\Comp$, $s\in G$ we can see that the map
    $$\h_\Comp \ni Z \mapsto e^{d \lambda(Z) }f \in L^2(G)$$
is holomorphic by \cite[Theorem~3.31]{Rudin3}. Hence, we have $e^{i d \lambda(Y_j)} f \rightarrow e^{i d \lambda(Y)} f$ as $j \rightarrow \infty$ and
\begin{align*}
    u^\h _\Comp \big( s_j \exp(iY_j) \big)
    & = \la e^{i d \lambda(Y_j)} f , \lambda(s_j )^{-1} g \ra\\
    & \rightarrow \la e^{i d \lambda(Y)} f , \lambda(s)^{-1} g \ra = u^\h _\Comp \big( s \exp (iY) \big)
\end{align*}
as $j \rightarrow \infty$,
proving the continuity of $u^\h _\Comp$.

The uniqueness of such extension follows from \cite[Lemma~7.21]{Rudin3} applied to each left coset $s H_\Comp$, $s\in G$.
Finally, we repeat the same argument of Proposition \ref{prop-hol-ext} for the general case.
\end{proof}

\subsection{The Gelfand spectrum of the algebra $\check{\mathcal{C}}$}

Let us check condition (A2) in the Introduction for the algebra $\check{\mathcal{C}}$.

\begin{thm}\label{thm-spectrum-of-C}
Every element in the algebra $\check{\mathcal{C}}$ extends uniquely to an $H$-holomorphic function on $G^{\mathfrak{h}}$. Furthermore, the evaluation functionals at the points of $G^{\mathfrak{h}}$ are precisely the characters of $\check{\mathcal{C}}$, i.e.,
    $${\rm Spec}\, \check{\mathcal{C}} \cong G^{\mathfrak{h}}.$$
\end{thm}
\begin{proof}
The first statement is from Theorem \ref{thm-partial-hol-ext}. Moreover, we can clearly see that the evaluation functionals at the points of $G^{\mathfrak{h}}$ belong to ${\rm Spec}\, \check{\mathcal{C}}$.

For the converse, we consider $\varphi \in {\rm Spec}\, \check{\mathcal{C}}$.
For the ``check map"
$$J:\mathcal{C}\to \check{\mathcal{C}},\; f\mapsto \check{f}$$
we know that the linear functional $\varphi\circ J\circ \Psi_*$ is multiplicative with respect to pointwise multiplication on $\g\cong \Real^d$. 
Composing the Fourier transform, we obtain a continuous linear functional $\psi = \varphi\circ J\circ \Psi_* \circ \Fc^{\Real^d}$ on $C^\infty_c(\Real^k) \odot \mathcal{S}(\Real^{d-k})$, which is multiplicative with respect to convolution. Restricting $\psi$ to the subspace $C^\infty_c(\Real^k) \odot C^\infty_c(\Real^{d-k})$ leads to the Cauchy functional equation as in \cite[Section 6.1.2]{GLLST}. We may apply the same argument of \cite[Theorem 6.13]{GLLST} to get $(c_1,\cdots,c_d) \in \Comp^d$ such that
    $$(\psi, \widehat{F}) = \int_{\Real^d}\widehat{F}(t_1,\cdots,t_d)e^{-i(c_1t_1+\cdots c_dt_d)}\,dt_1\cdots dt_d, \, \widehat{F} \in C^\infty_c(\Real^k) \odot C^\infty_c(\Real^{d-k}).$$
By putting non-zero functions $h_1,\cdots, h_k\in C^\infty_c(\Real)$ for the first $k$ components, we get a continuous linear functional $\phi$ on $\mathcal{S}(\Real^{d-k})$ of the form
$$(\phi, \widehat{G}) = \int_{\Real^{d-k}}\widehat{G}(t_{k+1},\cdots,t_d)e^{-i(c_{k+1}t_{k+1}+\cdots c_dt_d)}\,dt_{k+1}\cdots dt_d, \; \widehat{G} \in C^\infty_c(\Real^{d-k}).$$
Since $C^\infty_c(\Real^{d-k})$ is dense in $\mathcal{S}(\Real^{d-k})$ we can conclude that the function $e^{-i(c_{k+1}t_{k+1}+\cdots c_dt_d)}$ is a tempered distribution on $\Real^{d-k}$, which means that
$$(c_{k+1},\cdots,c_d) \in \Real^{d-k}.$$
Therefore, we have $$ (\varphi\circ J\circ \Psi_*, F) = (\psi, \widehat{F}) = F_\Comp(c_1,\cdots,c_d),\; F\in \Fc^{\Real^d}(C^\infty_c(\Real^k) \odot \mathcal{S}(\Real^{d-k})),$$
where $(c_1,\cdots,c_d) \in \mathfrak{h}_\Comp\times (\g\ominus \mathfrak{h})$ and
$F_\Comp$ is the unique holomorphic extension of $F$ for the first $k$ variables.
This, in turn, tells us that $\varphi\circ J$ is the evaluation on the point $\Psi_\Comp(c_1,\cdots,c_d) \in \{\exp(iY)s:s\in G, Y\in \mathfrak{h}\}$. Finally, we can conclude that $\varphi$ is the evaluation on the point $\Psi_\Comp(c_1,\cdots,c_d)^{-1} \in G^{\mathfrak{h}}$.
\end{proof}

Combining Theorem \ref{thm-B-embeds-into-A}, Proposition \ref{prop-properties-of-tilde-domain} and Theorem \ref{thm-spectrum-of-C} we get the following, which implies that the subalgebra $\check{\mathcal{C}}$ satisfies all the conditions (A1), (A2') and (A3") in the Introduction.

\begin{cor}\label{cor-nilpotent-main}
The topological algebra $\check{\mathcal{C}}$ is continuously embedded in $A(G,W)$ with a dense image for any weight $W = \iota(W_H)$ extended from a weight $W_H$ on the dual of $H$ as in Section \ref{subsec-weights}.
Furthermore, we have ${\rm Spec}\, \check{\mathcal{C}} \cong G^{\mathfrak{h}}$ and every element in ${\rm Spec}A(G,W)$ acts on $\check{\mathcal{C}}$ as an evaluation functional at the points of $G^\h$.
\end{cor}

\begin{rem}\label{rem-indep-proof-spec}
The above corollary provides a proof of Theorem \ref{thm-spectrum-extended-weight-general}, which is independent of the results in \cite{GT}, for a connected, simply connected and nilpotent Lie group $G$. 
In fact, Corollary \ref{cor-nilpotent-main} tells us that ${\rm Spec}A(G,W) \subseteq {\rm Spec}\, \check{\mathcal{C}} \cong G^{\mathfrak{h}}$.
Therefore, any element $\varphi \in {\rm Spec}A(G,W)$ is associated with an evaluation functional at the points of $G^\h$ so that \eqref{eq-hol-ext-nilpotent} tells us that $\varphi = \lambda(s) e^{i \partial \lambda(Y)}W^{-1}$ for some $s \in G$ and $Y \in \mathfrak{h}$.
Now we repeat the same argument of the proof of \cite[Theorem 6.17]{GLLST}.
By \cite[Proposition 2.1]{GLLST} we know that the operator $e^{i \partial \lambda(Y)}W^{-1}$ is bounded if and only if $e^{i \partial \lambda_H(Y)}W^{-1}_H$ is bounded, which leads to the conclusion we wanted.

\end{rem}

\section{The case of the $ax+b$-group}
In this final section, we examine the case of the $ax+b$-group $G$, where the global solution can still be established while $G$ does not belong to the class of nilpotent Lie groups.

We will follow the description of $G$ from \cite[Section 4.3]{FL15}. Let
$$G = \{(a,b): a,b\in \Real\}$$
with the group law
$$(a,b)\cdot (a',b') := (a+a', e^{-a'}b+b'),\;\; (a,b), (a',b') \in G.$$
Note that $(0,0)$ is the identity of $G$. The Lie algebra $\g$ is the vector space $\Real^2$ with the basis $\{A = (1,0), B=(0,1)\}$ and the Lie bracket
$$[A,B]=B.$$
The $ax+b$-group $G$ is a solvable Lie group whose exponential map
$$\exp: \g \to G,\;\;(s,t)\mapsto (s, \frac{e^{-s}-1}{-s}t)$$
is a global diffeomorphism. Note that we take the limit $s\to 0$ in the above for $\exp(0,t) = (0,t)$, $t\in \Real$.

The left Haar measure on $G$ can be chosen to be the Lebesgue measure $dadb$ so that $L^2(G) \cong L^2(\Real^2)$, and
the left regular representation $\lambda$ can be easily described as follows.
$$\lambda(a,b)F(s,t) = F(s-a, t-e^{a-s}b),\;\; F\in L^2(G),\; a,b,s,t \in \Real.$$
The unitary dual $\widehat{G}$ is known to be a set of two points consisting of representations $\pi_+$ and $\pi_-$ both acting on $L^2(\Real)$ as follows.
$$\pi_\pm(a,b)\xi(t) = \exp(\mp i be^{a-t}) \xi(t-a),\;\; \xi \in L^2(\Real),\; a,b\in \Real.$$
Furthermore, for $F\in L^1(G)$ we have
\begin{equation}\label{eq-pi-pm}
\pi_\pm(F)\xi(t) = \int_\Real \hat{F}^2(t-s, \pm e^{-s})\xi(s) \,ds,
\end{equation}
where $\displaystyle \hat{F}^2(s,u) = \frac{1}{2\pi}\int_\Real F(s,a)e^{-iau}\,da$.

Since $G$ is non-unimodular, we need the following Duflo-Moore operator $\delta$ acting on $L^2(\Real)$ (as a positive unbounded operator) for the complete Plancherel picture.
\begin{equation}\label{eq-DM-op}
\delta \xi(t) := e^{-\frac{t}{2}}\xi(t),\;\; \xi \in L^2(\Real).
\end{equation}
Note that we have an onto unitary (\cite[Section 4.3.2]{FL15})
\begin{equation}\label{eq-Plancherel-L2}
\Fc^G: F \in L^2(G) \mapsto \pi_+(F)\circ \delta \oplus \pi_-(F)\circ \delta \in S^2(L^2(\Real))\oplus_2 S^2(L^2(\Real))
\end{equation}
and an onto isometry (\cite[(3.50) and Theorem~4.15]{Fuhr})
\begin{equation}\label{eq-Plancherel-L1}
F \in A(G)\mapsto \pi_+(F)\circ \delta^2 \oplus \pi_-(F)\circ \delta^2 \in S^1(L^2(\Real))\oplus_1 S^1(L^2(\Real)).
\end{equation}
Here, $S^1(\Hi)$ and $S^2(\Hi)$ are the trace class and the Hilbert-Schmidt class on the Hilbert space $\Hi$ and $\oplus_p$ ($1\le p$) refers to the $\ell^p$-direct sum of Banach spaces.
The unitary $\Fc^G$ above intertwines the left regular representation $\lambda$ and the representation $\pi_+ \oplus \pi_-$.
This allows us to transfer the weight $W_\beta = e^{\beta |\partial\lambda(A)|}$ (for $\beta>0$) as in Section \ref{subsec-weights} acting on $L^2(G)$ to the direct sum $W^+_\beta \oplus W^-_\beta$, where the operators $W^\pm_\beta$ acting on $L^2(\Real)$ are given by
$$\Fc^\Real \circ W^\pm_\beta \circ (\Fc^\Real)^{-1} = M_{e^{-\beta |u|}}.$$

Now we define the subalgebra we are looking for.

\begin{defn}
Let $\mathcal{D}$ be the Fr\'{e}chet space given by
$$\mathcal{D} := \{F\in L^2(\Real^2)\cong L^2(G): e^{\gamma |u|}\widehat{F}(u,v) \in L^2(\Real^2),\; \forall \gamma>0\}$$
equipped with the semi-norms
$$\|F\|_\gamma := \|e^{\gamma |u|}\widehat{F}(u,v)\|_{L^2(\Real^2)},\;\; \gamma>0.$$
We also consider the space
$$\mathcal{E} := \Fc^{\Real^2}(C^\infty_c(\Real) \odot \mathcal{S}(\Real)).$$
The locally convex topological algebra structure of $\mathcal{E}$ comes from, through the map $\Fc^{\Real^2}$, the algebraic tensor product $C^\infty_c(\Real) \odot \mathcal{S}(\Real)$, which is an algebra with respect to convolution, so that $\mathcal{E}$ is an algebra with respect to pointwise multiplication.
\end{defn}

\begin{prop}\label{prop-density-D and E}
\begin{enumerate}
    \item The space $\mathcal{D}$ is continuously embedded in $A(G,W_\beta)$ with a dense image for all $\beta>0$.
    \item The algebra $\mathcal{E}$ is continuously embedded in $\mathcal{D}$ with a dense image.
\end{enumerate}
\end{prop}
\begin{proof}
(1) From the definition of $A(G,W_\beta)$ and \eqref{eq-Plancherel-L1} it suffices to show that the space
$$\tilde{\mathcal{D}} := \{W^+_\beta \pi_+(F)\circ \delta^2 \oplus W^-_\beta \pi_-(F)\circ \delta^2: F\in \mathcal{D}\}$$
is continuously embedded in $S^1(L^2(\Real))\oplus_1 S^1(L^2(\Real))\cong A(G)$ with a dense image. In fact, for $F\in \mathcal{D}$ we combine \eqref{eq-pi-pm} and \eqref{eq-DM-op} to get
$$(\pi_\pm(F) \delta^2) \xi(t) = \int_\Real \hat{F}^2(t-s, \pm e^{-s})e^{-s}\xi(s) \,ds$$
so that
$$(\Fc^\Real\circ W^\pm_\beta\pi_\pm(F) \delta^2)\eta(u) = \int_\Real e^{-\beta |u|}\widehat{F}(u, \pm e^{-s})e^{- ius}e^{-s}\eta(s) \,ds,$$
for $\xi, \eta \in L^2(\Real)$.
Note that the latter operator
is an integral operator with the kernel function
$$K(u,s) = e^{-\beta |u|}\widehat{F}(u, \pm e^{-s})e^{- ius}e^{-s}.$$
Taking $e^{-s} = v>0$ we get $\mp e^{-\beta |u|}\widehat{F}(u, \pm v)v^{- iu}$, which is clearly an $L^2$-function. Thus, the operator $W^\pm_\beta\pi_\pm(F) \delta^2$ is Hilbert-Schmidt, and consequently a trace-class operator. This explains the continuous embedding of $\tilde{\mathcal{D}}$ in $A(G)$.

For density we consider a dense subspace of $S^1(L^2(\Real))\oplus_1 S^1(L^2(\Real))$ consisting of all finite rank operators and show that it is contained in the space $\tilde{\mathcal{D}}$.
Recall the sequence $(\phi_n)_{n\ge 0}$ of the Hermite functions on $\Real$ and consider the rank 1 operator $T_{m,n}$ given by 
$T_{m,n}\xi = \la \xi, \phi_n\ra \phi_m$. Now it is enough to show that $T_{m,n} \oplus T_{m',n'} \in \tilde{\mathcal{D}}$ for any $m,n,m',n' \ge 0$.
Note that
$$(\Fc^\Real \circ T_{m,n})\eta = \la \eta, \phi_n\ra \Fc^\Real \phi_m = \frac{1}{\sqrt{2\pi}} (-i)^m\la \eta, \phi_n\ra \phi_m,$$
so that it is an integral operator with the kernel function
$$\tilde{K}(u,s) = \frac{ (-i)^m }{\sqrt{2\pi}}\phi_m(u)\phi_n(s).$$
Comparing the above two kernel functions $K$ and $\tilde{K}$ we get
$$e^{-\beta |u|}\widehat{F}(u, \pm e^{-s})e^{-ius}e^{-s} = \frac{ (-i)^m }{\sqrt{2\pi}}\phi_m(u)\phi_n(s).$$
Let us apply this observation for the first direct summand (or for the choice $+e^{-s}$ in the LHS) and for the second direct summand (or for the choice $-e^{-s}$ in the LHS) we use $m', n'\ge 0$ to get a function given by
$$\widehat{F}(u, v) = \begin{cases}
\frac{(-i)^m}{\sqrt{2 \pi}}\phi_m(u)\phi_n(-\log v)e^{\beta |u|}v^{- iu + 1}, & v>0\\
\frac{(-i)^{m'}}{\sqrt{2 \pi}}\phi_{m'}(u)\phi_{n'}(-\log (-v))e^{\beta |u|}(-v)^{ - iu + 1}, & v<0\\
0, & v=0
\end{cases}.$$
Note that the super-exponential decay of Gaussian functions guarantees that the above expression defines a continuous function $\widehat{F} \in L^2(\Real^2)$. 
Indeed, the term $|\phi_n(-\log v)v^{- iu + 1}| = |\phi_n(-\log v)|\cdot |v|$ requires some care and we have
$$v\cdot e^{-(\log v)^2/2}\le v\cdot \sqrt{e^{-6\log v}} = v^{-2},\;\;v>e^6$$
and clearly
$$\lim_{v\to 0^+}(v\cdot e^{-(\log v)^2/2}) = 0.$$
Furthermore, the super-exponential decay of Gaussian functions again tells us that $e^{\gamma |u|}\widehat{F}(u, v) \in L^2(\Real^2)$ for any $\gamma>0$, which means that $F\in \mathcal{D}$.

\vspace{0.5cm}

(2) A typical smooth approximation procedure gives us the conclusion we wanted.
\end{proof}

\begin{cor}
Let $W$ be a weight on the dual of $G$ extended from the subgroup $H = \{(a,0): a\in \Real\} \cong \Real$ as in Section \ref{subsec-weights}.
The topological algebra $\mathcal{E}$ is continuously embedded in $A(G,W)$ with a dense image and satisfies ${\rm Spec}\, \mathcal{E} \cong G^{\mathfrak{h}}\cong \Comp \times \Real$.
Furthermore, every element of $\mathcal{E}$ has a unique $H$-holomorphic extension to $G^{\mathfrak{h}}$.
Finally, every element in ${\rm Spec}A(G,W)$ acts on $\mathcal{E}$ as an evaluation functional at the points of $G^\h$.
\end{cor}
\begin{proof}
For the density of $\mathcal{E}$ in $A(G,W)$ we appeal to Proposition \ref{prop-weight-comparison} and \eqref{eq-domination} to narrow down to the case of $W_\beta$, which can be covered by Proposition \ref{prop-density-D and E}.
For the unique $H$-holomorphic extension, we only need to note that $G^{\mathfrak{h}}\cong \Comp \times \Real$ and recall the Paley-Wiener theorem.
For the conclusion ${\rm Spec}\, \mathcal{E} \cong G^{\mathfrak{h}}$ and the last statement, we may repeat the proof of Theorem \ref{thm-spectrum-of-C}.
\end{proof}

\begin{rem}
The above corollary provides a proof of Theorem \ref{thm-spectrum-extended-weight-general}, which is independent of the results in \cite{GT}, for the $ax+b$-group $G$ using the same argument as in Remark \ref{rem-indep-proof-spec}.
\end{rem}

\emph{Acknowledgements}: H.H. Lee was supported by the Basic Science Research Program through the National Research Foundation of Korea (NRF) Grant NRF-2017R1E1A1A03070510 and the National Research Foundation of Korea (NRF) Grant funded by the Korean Government (MSIT) (Grant No.2017R1A5A1015626).




\bibliographystyle{abbrv}
\bibliography{ref}

\begin{thebibliography}{10}

\bibitem{BtD}
T.~Br{\"o}cker and T.~tom Dieck.
\newblock {\em Representations of compact Lie groups}.
\newblock Graduate Texts in Mathematics, 98. Springer-Verlag, New York, 1985.

\bibitem{Fuhr}
H.~F\"{u}hr.
\newblock {\em Abstract Harmonic Analysis of Continuous Wavelet Transforms}.
\newblock Lecture Notes in Mathematics. Springer Berlin, Heidelberg, 1 edition,
  2005.

\bibitem{FL15}
H.~Fujiwara and J.~Ludwig.
\newblock {\em Harmonic analysis on exponential solvable Lie groups}.
\newblock Springer Monogr. Math. Springer, 2015.

\bibitem{GLLST}
M.~Ghandehari, H.~H. Lee, J.~Ludwig, N.~Spronk, and L.~Turowska.
\newblock Beurling-fourier algebras on lie groups and their spectra.
\newblock {\em Advances in Mathematics}, 391:107951, 2021.

\bibitem{GT}
O.~Giselsson and L.~Turowska.
\newblock Beurling-fourier algebras and complexification, 2022.

\bibitem{Goodman_entire}
R.~Goodman.
\newblock Analytic and entire vectors for representations of lie groups.
\newblock {\em Transactions of the American Mathematical Society}, 143:55--76,
  1969.

\bibitem{Goodman_nilpotent}
R.~Goodman.
\newblock Complex fourier analysis on a nilpotent lie group.
\newblock {\em Transactions of the American Mathematical Society},
  160:373--391, 1971.

\bibitem{HH08}
D.~D. Haroske and H.~Triebel.
\newblock {\em Distributions, Sobolev spaces, elliptic equations}.
\newblock EMS Textbooks in Mathematics. European Mathematical Society (EMS),
  2008.

\bibitem{Hor66}
J.~Horvath.
\newblock {\em Topological Vector Spaces and Distributions, Vol. 1}.
\newblock Addison-Wesley, 1966.

\bibitem{Kaniuth}
E.~Kaniuth and A.~T.-M. Lau.
\newblock {\em Fourier and Fourier-Stieltjes Algebras on Locally Compact
  Groups}.
\newblock Mathematical Surveys and Monographs. American Mathematical Society,
  2018.

\bibitem{Knapp}
A.~W. Knapp.
\newblock {\em Lie Groups: Beyond an Introduction}.
\newblock Progress in Mathematics. Birkhäuser Boston, MA, 2002.

\bibitem{Kothe83}
G.~Köthe.
\newblock {\em Locally Convex Spaces. Fundamentals}.
\newblock Grundlehren der mathematischen Wissenschaften. Springer, Berlin,
  Heidelberg, 1 edition, 1983.

\bibitem{CG90}
P.~G. L.~J.~Corwin.
\newblock {\em Representations of nilpotent Lie groups and their applications.
  Part I. Basic theory and examples.}
\newblock Cambridge Studies in Advanced Mathematics, 18. Cambridge University
  Press, 1990.

\bibitem{LS12}
H.~H. Lee and E.~Samei.
\newblock Beurling-fourier algebras, operator amenability and arens regularity.
\newblock {\em J. Funct. Anal.}, 262:167--209, 2012.

\bibitem{LST12}
J.~Ludwig, N.~Spronk, and L.~Turowska.
\newblock Beurling-fourier algebras on compact groups: spectral theory.
\newblock {\em J. Funct. Anal.}, 262:463--499, 2012.

\bibitem{LT06}
J.~Ludwig and L.~Turowska.
\newblock Growth and smooth spectral synthesis in the fourier algebras of lie
  groups.
\newblock {\em STUDIA MATHEMATICA}, 176, 2006.

\bibitem{McK2}
K.~McKennon.
\newblock The complexification and differential structure of a locally compact
  group.
\newblock {\em Transactions of the American Mathematical Society},
  267:237--258, 1981.

\bibitem{Rudin3}
W.~Rudin.
\newblock {\em Functional Analysis}.
\newblock International Series in Pure and Applied Mathematics. McGraw-Hill, 2
  edition, 1991.

\bibitem{SchWol99}
H.~H. Schaefer and M.~P. Wolff.
\newblock {\em Topological Vector Spaces}.
\newblock Graduate Texts in Mathematics. Springer New York, NY, 2 edition,
  1999.

\bibitem{Schm12}
K.~Schmüdgen.
\newblock {\em Unbounded Self-adjoint Operators on Hilbert Space}.
\newblock Graduate Texts in Mathematics. Springer Dordrecht, 1 edition, 2012.

\bibitem{Takesaki2}
M.~Takesaki.
\newblock {\em Theory of Operator Algebras II}.
\newblock Encyclopaedia of Mathematical Sciences. Springer Berlin, Heidelberg,
  1 edition, 2003.

\bibitem{Taylor_harmonic}
M.~E. Taylor.
\newblock {\em Noncommutative Harmonic Analysis}.
\newblock Mathematical Surveys and Monographs. American Mathematical Society,
  1986.

\bibitem{Th75}
G.~E.~F. Thomas.
\newblock Integration of functions with values in locally convex suslin spaces.
\newblock {\em Transactions of the American Mathematical Society}, 212:61--81,
  1975.

\bibitem{Wa71}
L.~Waelbroeck.
\newblock {\em Topological Vector Spaces and Algebras}.
\newblock Lecture Notes in Mathematics. Vol. 230. Springer Berlin, Heidelberg,
  1 edition, 1971.

\end{thebibliography}

\end{document}